\title{\textsc{\textbf{{
The spinorial energy for asymptotically Euclidean Ricci flow}}}}
\author{\textsc{Julius Baldauf\thanks{Supported in part by the National Science Foundation. {\it E-mail}: \texttt{juliusbl@mit.edu}} 
\quad \quad \quad 
Tristan Ozuch}
\vspace{0.2cm}\\
    \textsc{\footnotesize MIT Department of Mathematics}\vspace{-0.1cm}\\
    \textsc{\footnotesize Cambridge, MA}
    \vspace{-0.05cm}
}
\date{}

\documentclass[11pt, letterpaper, leqno, twoside]{article}

\usepackage[colorinlistoftodos,prependcaption]{todonotes}

\usepackage{amsthm,xcolor}
\usepackage[colorlinks, linkcolor=red]{hyperref}
\usepackage[top=1.5in, lmargin=1in, rmargin=1in, marginparwidth=0.9in]{geometry}
\usepackage[english]{babel}
\usepackage{amssymb}
\usepackage{mathrsfs}
\usepackage{mathtools}
\usepackage{bbm}
\usepackage{amsthm}
\usepackage{fancyhdr}
\usepackage{tikz-cd}
\usepackage{comment}
\usepackage{etoolbox}
\usepackage{breqn}
\usepackage{esint}
\usepackage{appendix}
\hypersetup{
    colorlinks=true,
    linkcolor=blue,
    filecolor=magenta,      
    urlcolor=cyan,
	citecolor=blue
}

\usepackage[T1]{fontenc}
\usepackage[center]{titlesec}
\usepackage{xcolor}
\usepackage[bottom]{footmisc}
\usepackage{indentfirst}
\usepackage{xpatch}
\usepackage{chngcntr}

\makeatletter
\renewcommand\th@plain{\slshape}
\xpatchcmd{\proof}{\itshape}{\slshape}{}{}
\makeatother

\makeatletter
\renewcommand\th@plain{\slshape}
\makeatother

\titlelabel{\thetitle.\quad}
\titleformat*{\section}{\centering\large\scshape\sffamily}
\titleformat{\subsection}[runin]
  {\normalfont\bfseries}{\thesubsection.}{0.6em}{}
\titleformat{\subsubsection}[runin]
  {\normalfont\bfseries}{\thesubsubsection.}{0.6em}{}

\fancypagestyle{mypagestyle}{%
    \fancyhf{}
    
    \fancyhead[OC]{\small Julius Baldauf \qquad \qquad \qquad \qquad Tristan Ozuch}
    \fancyhead[EC]{\small Spinorial energy for AE Ricci flow}
    \fancyhead[EL]{\thepage}
    \fancyhead[OR]{\thepage}
    \setlength{\headheight}{13.6pt}
}
\numberwithin{equation}{section}

\theoremstyle{plain} 
\newtheorem{lemma}[equation]{Lemma}

\newtheorem{proposition}[equation]{Proposition}
\newtheorem{theorem}[equation]{Theorem}
\newtheorem{corollary}[equation]{Corollary}

\theoremstyle{definition}

\newtheorem{remark}[equation]{Remark}

\newcommand{\R}{\mathbb{R}}

\newcommand{\N}{\mathbb{N}}

\renewcommand{\tilde}{\widetilde}
\newcommand{\Div}{\mathrm{div}}

\newcommand{\Ric}{\mathrm{Ric}}

\newcommand{\Scal}{\mathrm{R}}

\renewcommand{\phi}{\varphi}
\renewcommand{\epsilon}{\varepsilon}

\newcommand{\Rea}{\,\mathrm{Re}\,}

\newcommand{\tr}{\mathrm{tr}}
\newcommand{\Hess}{\mathrm{Hess}}
\newcommand{\mass}{\mathfrak{m}}
\newcommand{\intprod}{\lrcorner\,}
\newcommand{\euc}{\mathrm{euc}}
\newcommand{\dL}{\Delta}
\newcommand{\Ld}{\mathscr{L}}

\newcommand{\Cyl}{\bar{M}}
\newcommand{\nablaCyl}{\overline{\nabla}}

\begin{document}

\maketitle

\begin{abstract}
This paper introduces a functional generalizing Perelman's weighted Hilbert-Einstein action and the Dirichlet energy for spinors. It is well-defined on a wide class of non-compact manifolds; 
on asymptotically Euclidean manifolds, the functional is shown to admit a unique critical point, which is necessarily of min-max type, and Ricci flow is its gradient flow.
The proof is based on variational formulas for weighted spinorial functionals, valid on all spin manifolds with boundary.
\end{abstract}

\section{Introduction}

Spinors are vectors in a complex vector space canonically associated with Euclidean space. They were first discovered by \'Elie Cartan a century ago \cite{C}, and soon thereafter, Paul Dirac \cite{D} used them to model the behavior of electrons and other elementary particles.
Spinors have since then been used fruitfully in mathematics to understand the geometry and topology of static manifolds \cite{AS, W1, GL, W2}. 
This paper introduces spin geometry into the Ricci flow  \cite{Ha} by showing that it is the gradient flow of a natural spinorial functional on asymptotically Euclidean (AE) manifolds. 

The gradient flow formulation established here is the analogue of Perelman's entropy monotonicity on closed manifolds \cite{P}.
Perelman showed that Ricci flow on closed manifolds is the gradient flow of the $\lambda$-entropy, which is proportional to the first eigenvalue of the Schr\"odinger operator $-\Delta +\frac{1}{4}\Scal$ acting on \emph{functions}. Due to Kato's inequality and the Lichnerowicz formula, the $\lambda$-entropy is bounded above by the first eigenvalue of the square of the Dirac operator $D^2=-\Delta +\frac{1}{4}\Scal$ acting on \emph{spinors}. This bound suggests a link between Perelman's $\lambda$-entropy and the Dirac operator.

The link is provided by the weighted Dirac operator \cite{BO}, which is the natural generalization of the Atiyah-Singer Dirac operator for a weighted spin manifold $(M^n,g,e^{-f})$. It is defined as
\begin{equation}\label{eqn: weighted Dirac operator}
    D_f=D-\frac{1}{2}(\nabla f)\cdot,
\end{equation}
where $\nabla f$ acts by Clifford multiplication, and $D$ denotes the standard (unweighted) Dirac operator. 
First introduced by Perelman \cite{P}, 
the weighted Dirac operator 
is self-adjoint with respect to the weighted measure $e^{-f}dV$, is unitarily equivalent to the standard Dirac operator, 
and satisfies the weighted Lichnerowicz formula (\ref{eqn: weighted Lichnerowicz}) involving Perelman's weighted scalar curvature 
\begin{equation}
    \Scal_f=\Scal+2\Delta f-|\nabla f|^2.
\end{equation}

On a weighted, asymptotically Euclidean (AE), spin manifold, the weighted Dirac operator allows for the generalization of a Witten spinor: a weighted Witten spinor is a spinor lying in the kernel of the weighted Dirac operator and which is asymptotic to a constant spinor of unit norm.
The weighted Dirichlet energy of a weighted Witten spinor plays the role of Perelman's $\lambda$-entropy for closed manifolds because Ricci flow is the gradient flow of this weighted Dirichlet energy for a certain weight \cite{BO}. 
Here it is shown that the coupled elliptic system consisting of the weighted scalar-flat equation and the weighted Witten spinor equation has a natural variational interpretation.

On an AE, spin manifold $(M^n,g)$, define the energy functional $\mathscr{E}_g$ depending on a spinor $\psi$, asymptotic to a constant spinor of norm 1, and a weight function $f$, asymptotic to 0 at infinity, by
\begin{equation}\label{eqn: energy functional}
    \mathscr{E}_g(\psi,f)=\int_M\Big( 4|\nabla \psi|^2+\Scal_f(|\psi|^2-1)\Big) e^{-f}dV_g.
\end{equation}
This energy generalizes various well-known functionals, including Perelman's weighted Hilbert-Einstein action, 
the ``spinorial energy'' \cite{AWW}, 
and the weighted Dirichlet energy of the spinor; see Section \ref{sec: critical points}.
For suitable choices of the spinor and weight, the value of the energy (\ref{eqn: energy functional}) equals the difference between the ADM mass and the Hilbert-Einstein action, also known as the Regge-Teitelboim Hamiltonian,
or the difference between the weighted ADM mass and Perelman's weighted Hilbert-Einstein action \cite{BO,DO,DO2}. 
The energy functional introduced here thus provides a unified treatment of many important functionals in geometric analysis and physics.

The following theorem characterizes the critical points of the energy (\ref{eqn: energy functional}).

\begin{theorem}
[Critical points]
\label{thm: existence and uniqueness of critical points, intro}
On every spin, AE manifold with nonnegative scalar curvature, the functional $\mathscr{E}_g$ admits a unique critical point $(\psi_g,f_g)$. This critical point satisfies the elliptic equations
\begin{equation}\label{eqn: critical point equations}
    \Scal_{f_g}=0 \qquad \text{and} \qquad D_{f_g}\psi_g=0,
\end{equation}
so $\psi_g$ is an $f_g$-weighted Witten spinor. Moreover, $(\psi_g,f_g)$ is a min-max critical point,
\begin{equation}\label{eqn: min-max quantity}
    \mathscr{E}_g(\psi_g,f_g)=\max_{f} \min_{\psi}\mathscr{E}_g(\psi,f).
\end{equation}
\end{theorem}

Given this theorem, define $\kappa(g)$ to be the energy of the unique critical point $(\psi_g,f_g)$ of $\mathscr{E}_g$,
\begin{equation}
    \kappa(g)
    =\max_{f} \min_{\psi}\mathscr{E}_g(\psi,f) 
    =\int_M|\nabla \psi_g|^2 \,e^{-f_g}dV.
\end{equation}
The main theorem of this paper concerns the time derivative of $\kappa(g(t))$ along a Ricci flow $\partial_t g = -2 \Ric(g)$. 
Because the metric along a Ricci flow is changing in time, the spin bundle is also changing, though the spin bundles at different times are isomorphic. 
A standard method for dealing with this subtlety is to consider the spacetime cylinder $M\times I$, equipped with a certain cylindrical metric.
The spin bundle of the cylinder can then be related to the spin bundles along the Ricci flow.
With this identification of the spin bundles understood, the main theorem of this paper shows that the energy $\kappa$ is the analogue of Perelman's $\lambda$-entropy, in the sense that Ricci flow is its gradient flow.

\begin{theorem}
[Monotonicity]\label{thm: monotonicity intro}
On every spin, AE Ricci flow with nonnegative scalar curvature, there exists at each time a unique min-max critical point $(\psi,f)$ of $\mathscr{E}_g$ and
\begin{equation}\label{eqn: weighted Dirichlet monotonicity}
    \kappa'(t)
    =\frac{d}{dt}\int_M|\nabla \psi|^2\,e^{-f}dV
    =-\frac{1}{2}\int_M|\Ric+\Hess_{f}|^2\,e^{-f}dV.
\end{equation}
In particular, Ricci flow is the $L^2(e^{-f}dV)$-gradient flow of $\kappa$, the weighted Dirichlet energy of $\psi$.
\end{theorem}

Note that the right-hand-side of the monotonicity formula (\ref{eqn: weighted Dirichlet monotonicity}) is independent of the spinor. This fact may be interpreted as a parabolic analogue of Witten's formula, which expresses the ADM mass in terms of a ``test spinor,'' even though the ADM mass may be defined without reference to any spinor. 
The reason is that if the spinor solves (\ref{eqn: critical point equations}), then its weighted Dirichlet energy equals a boundary term at infinity, which is independent of the spinor by the boundary conditions.
The monotonicity formula (\ref{eqn: weighted Dirichlet monotonicity}) is proven here via the first variation of the weighted Dirichlet energy for spinors. 
While monotonicity was recently proven in \cite{BO} via an indirect argument relying on the results of Deruelle-Ozuch \cite{DO}, the proof given here is independent of said results.

The weighted variational formulas derived here are also of independent interest. For example, they imply that the ADM mass of a spin, AE manifold with nonnegative scalar curvature is constant along Ricci flow.
Constancy of the ADM mass along Ricci flow was previously proven by different means \cite{DM, L}. Here, a proof is given using Witten's formula for the mass. 

\begin{theorem}
[Constancy of mass]\label{thm: Constancy of ADM mass}
The ADM mass is constant along every spin, AE Ricci flow with nonnegative scalar curvature.
\end{theorem}

Furthermore, the weighted variational formulas derived here generalize those from the unweighted case, which have recently received much attention: the gradient flow of the (unweighted) Dirichlet energy for spinors, introduced by Ammann-Weiss-Witt \cite{AWW}, is equivalent to a modified Ricci flow coupled to a spinor evolving parabolically in time \cite{HW} (see also \cite{AWW2, S, CP}).
Additionally, the weighted variational formulas derived here are valid on all manifolds with boundary, so the techniques developed here are expected to extend to other geometries adapted to spin methods, such as asymptotically hyperbolic manifolds \cite {Wang} and ALF manifolds \cite{Min}. 

The paper is organized as follows: Sections \ref{sec: Spin geometry of generalized cylinders} and \ref{sec: Weighted Dirac operator} give the necessary background on spin geometry on evolving manifolds and on the weighted Dirac operator. Section \ref{sec: Variational formulas} derives variational formulas for natural weighted, spinorial quantities. Section \ref{sec: Monotonicity} applies said formulas to prove the monotonicity theorems. Appendix \ref{sec: Time derivatives of weighted Witten spinors} proves the existence and regularity of time derivatives of weighted Witten spinors, and Appendix \ref{sec: weighted integration by parts} contains useful weighted integration by parts formulas.

\section*{\small\bf Acknowledgements}
The first author thanks W. Minicozzi for continual support, as well as B. Ammann, T. Ilmanen, and C. Taubes for useful discussions. Part of this work was completed while the first author was funded by a National Science Foundation Graduate Research Fellowship.

\section{Spinors on evolving manifolds}\label{sec: Spin geometry on evolving manifolds}

\subsection{Spin geometry of generalized cylinders.}\label{sec: Spin geometry of generalized cylinders}

The spin bundle, and hence the Dirac operator, depends on a choice of Riemannian metric. For two choices of Riemannian metrics, the spin bundles are isomorphic, though in general not canonically so. 
Given a 1-parameter family of Riemannian metrics, there does exist a natural identification of the spin bundles at different times, obtained via the generalized cylinder construction of \cite[\S 3 -- \S 5]{BGM}. 
This section recalls the generalized cylinder construction and the associated variational formulas, which are applied in the later sections to the special context of Ricci flows. 
The notation established here is used in the remainder of the paper.

Let $M$ be a smooth $n$-manifold admitting a spin structure, and let $(g_t)_{t\in I}$ be a 1-parameter family of Riemannian metrics on $M$ whose time derivative is denoted
\begin{equation}\label{eqn: variation of metric}
    \partial_tg=\dot{g}.
\end{equation}
Corresponding to this 1-parameter family, define the {\it generalized cylinder} by 
\begin{equation}
    \Cyl:=I\times M,
\end{equation}
equipped with the Riemannian metric 
\begin{equation}
    \bar{g}:=dt^2+g_t.
\end{equation}
For $t\in I$, abbreviate the Riemannian manifold $(M,g_t)$ by $M_t$, and sometimes simply by $M$ when the choice of $t$ is clear from the context. Connections associated with $(\Cyl,\bar{g})$ are denoted $\nablaCyl$, while those associated with $(M_t,g_t)$ are denoted $\nabla^{g_t}$, or simply $\nabla$ when the choice of $t$ is clear from the context.
The vector field $\partial_t$ on $\Cyl$ is normal to $M_t$ and has unit $\bar{g}$-length. Moreover, its integral curves are geodesics, i.e.\ 
\begin{equation}\label{eqn: time derivative is geodesic vector field}
    \overline{\nabla}_{\partial_t}\partial_t=0.
\end{equation}
Let $W$ denote the Weingarten tensor with respect to the embedding $M_t\subset \Cyl$. This tensor is defined by the condition that the Levi-Civita connections of $\Cyl$ and $M$ are related by
\begin{equation}
    \nablaCyl_X Y=\nabla_XY+\langle W(X),Y\rangle\partial_t,
\end{equation}
for all vector fields $X,Y$ on $M$. 
From the Koszul formula for the Levi-Civita connection,
it follows 
that
\begin{equation}\label{eqn: Weingarten equals -1/2 dot(g)}
    \langle W(X),Y\rangle =-\frac{1}{2}\dot{g}(X,Y).
\end{equation}
Therefore, the Levi-Civita connections of $\Cyl$ and $M$ are related by 
\begin{equation}
    \nablaCyl_XY=\nabla_XY-\frac{1}{2}\dot{g}(X,Y)\partial_t,
\end{equation}
for all vector fields $X,Y$ on $M$. Moreover, computation of the Christoffel symbols 
of $\bar{g}$ with respect to a local orthonormal frame $(e_1,\dots,e_n)$ of the metric $g_t$ implies that for any vector field $X$ on $M$,
\begin{equation}\label{eqn: cylindrical time derivative}
    \nablaCyl_{\partial_t}X=\partial_tX+\frac{1}{2}\sum_{i=1}^n\dot{g}(X,e_i)e_i.
\end{equation}

Since $\Cyl$ is homotopy equivalent to $M$, spin structures on $\Cyl$ are in bijection with those on $M$. 
A spin structure on $\Cyl$ can be restricted to a spin structure on $M=M_t$ in the following way. 
Let $\Theta:P_{\mathrm{Spin}}(\Cyl)\to P_{\mathrm{SO}}(\Cyl)$ be a spin structure on $\Cyl$. 
Embed the bundle oriented orthonormal frames of $M$, $P_{\mathrm{SO}}(M)$, into the bundle of space and time oriented orthonormal frames of $\Cyl$ restricted to $M$, $P_{\mathrm{SO}}(\Cyl)|_M$, by the map $\iota:(e_1,\dots,e_n) \mapsto (\partial_t,e_1,\dots,e_n)$. 
Then $P_{\mathrm{Spin}}(M):=\Theta^{-1}(\iota(P_{\mathrm{SO}}(M)))$ defines a spin structure on $M$.
Conversely, given a spin structure on $M$, it can be pulled back to $\Cyl$ yielding a $\tilde{GL}^+(n,\R)$-principal bundle on $\Cyl$. 
Enlarging the structure group via the embedding $\tilde{GL}^+(n,\R)\hookrightarrow \tilde{GL}^+(n+1,\R)$ covering the standard embedding 
\begin{equation}
    GL^+(n,\R)\hookrightarrow GL^+(n+1,\R) \qquad\qquad a\mapsto 
    \begin{pmatrix}
    1 & 0 \\
    0 & a 
    \end{pmatrix},
\end{equation}
yields the spin structure on $\Cyl$ which restricts to the given spin structure on $M$ \cite[\S 3 -- \S 5]{BGM}. This paper always implicitly assumes this identification of spin structures on $M$ and $\Cyl$.

Clifford multiplication on $\Cyl$ is denoted by ``$\bullet$'', while Clifford multiplication on $M_t$ is denoted by ``$\cdot$''. Recall 
that, as Hermitian vector bundles over $M_t$, there is an isometry of the complex spin bundles $\Sigma \Cyl|_{M_t}=\Sigma M_t$ when $n$ is even, while $\Sigma^+\Cyl|_{M_t}=\Sigma M_t$ when $n$ is odd. In both cases the Clifford multiplications are related by 
\begin{equation}\label{eqn: relation between Clifford multplications}
    X\cdot\psi=\partial_t\bullet X\bullet \psi.
\end{equation}
When $n$ is odd, $\Sigma M_t$ is henceforth identified with $\Sigma^+\Cyl|_{M_t}$, so that (\ref{eqn: relation between Clifford multplications}) holds.

Let $\langle \cdot,\cdot\rangle$ be the spin metric on $\Sigma \Cyl$, that is, the unique Hermitian metric on $\Sigma \Cyl$ for which Clifford multiplication by unit vectors is unitary. This metric is compatible with the spin connection in the sense that
\begin{equation}
    X\langle \phi,\psi\rangle 
    =\langle \nablaCyl_X\phi,\psi\rangle+\langle \phi,\nablaCyl_X\psi\rangle
\end{equation}
holds for all vector fields $X$ on $\Cyl$. 
Combining Equations (\ref{eqn: Weingarten equals -1/2 dot(g)}) and (\ref{eqn: relation between Clifford multplications}) 
with the local expression of the spin connection 
\begin{equation}\label{eqn: spin connection in local coordinates}
    \nablaCyl_{i}\psi
    =\partial_i\psi +\frac{1}{4}\sum_{j,k=0}^n \bar{\Gamma}_{ij}^ke_j\bullet e_k\bullet \psi
\end{equation}
in an orthonormal frame $(e_0,\dots,e_n)$ of $T\Cyl$, with $e_0=\partial_t$, implies the following relationship 
between the spin connections of $\Cyl$ and $M$: for any vector field $X$ on $M$,
\begin{equation}\label{eqn: relation between spin connections}
    \nablaCyl_X\psi
    =\nabla_X\psi+\frac{1}{4} \sum_{i=1}^n\dot{g}(X,e_i)e_i \cdot \psi.
\end{equation}
\subsection{Weighted Dirac operator.}\label{sec: Weighted Dirac operator}

The remainder of this paper employs tools from the theory of spin geometry of weighted manifolds, developed in \cite[\S 1]{BO}. For the convenience of the reader and to establish the notation for what is to come, the relevant facts are reviewed below. 

A weighted spin manifold is a spin manifold $(M^n,g)$ equipped with a function $f:M\to \R$ defining the weighted measure $e^{-f}dV$.
The weighted Dirac operator $D_f:\Gamma(\Sigma M)\to \Gamma(\Sigma M)$ of a weighted spin manifold is defined as 
\begin{equation}
    D_f=D-\frac{1}{2}(\nabla f) \cdot \;,
\end{equation}
where $D=e_i\cdot \nabla_i$ is the standard Dirac operator, namely the composition of the spin covariant derivative $\nabla$ with Clifford multiplication.
(Throughout this paper, 1-forms and vector fields on time slices $M_t$ will often be identified via the metric $g_t$ without explicit mention.)
The weighted Dirac operator is the Dirac operator associated with the modified spin connection $\nabla^f:\Gamma(\Sigma M)\to \Gamma(T^*M\otimes \Sigma M)$, defined by
\begin{equation}\label{eqn: weighted spin connection}
    \nabla_X^f\psi=\nabla_X \psi-\frac{1}{2}(\nabla_X f)\psi.
\end{equation}
The modified spin connection $\nabla^f$ is {\it not} metric compatible with the standard metric \cite[Prop. 2.5]{BHM+} on the spin bundle, $\langle\cdot,\cdot\rangle$, however, it is compatible with the modified metric $\langle \cdot,\cdot\rangle_f := \langle \cdot,\cdot\rangle e^{-f}$, that is,
\begin{equation}\label{eqn: weighted connection is weighted metric compatible}
    X(\langle \psi,\phi\rangle e^{-f})
    =\langle \nabla^f_X\psi,\phi\rangle e^{-f}+\langle \psi,\nabla^f_X\phi\rangle e^{-f},
\end{equation}
for all vector fields $X$ and spinors $\psi,\phi$. Moreover, since Clifford multiplication is parallel with respect to the standard spin connection, it is also parallel with respect to $\nabla^f$. This means that
\begin{equation}\label{eqn: Cliff mult is weighted parallel}
    \nabla_X^f(Y\cdot \psi)=Y\cdot \nabla_X^f\psi +(\nabla_XY)\cdot \psi,
\end{equation}
for all vector fields $X,Y$ and spinors $\psi$.

The weighted Dirac operator satisfies the following weighted integration by parts formula on closed manifolds
\begin{equation}\label{eqn: weighted Laplcian IBP}
    \int_M\langle \psi,D_f\phi\rangle e^{-f}\,dV
    =\int_M \langle D_f \psi,\phi\rangle e^{-f}\,dV,
\end{equation}
and hence is self-adjoint on the weighted space $L_f^2=L^2(e^{-f}\,dV)$. 
Furthermore, $D_f$ satisfies a weighted Lichnerowicz formula, which was observed by Perelman \cite[Rem.\ 1.3]{P}. To state it, let
\begin{equation}\label{eqn: defn of weighted Laplacian}
    \dL_f=\Delta  -\nabla_{\nabla f}
\end{equation} 
be the weighted Laplacian acting on spinors and let
\begin{equation}\label{eqn: defn of weighted scalar curvature}
    \Scal_f=\Scal+2\Delta f-|\nabla f|^2
\end{equation}
be Perelman's weighted scalar curvature (or P-scalar curvature).
Then the square of the weighted Dirac operator $D_f$ satisfies
\begin{equation}\label{eqn: weighted Lichnerowicz}
    D_f^2=-\dL_f+\frac{1}{4}\Scal_f.
\end{equation}
Furthermore, the weighted (Bakry-\'Emery) Ricci curvature $\Ric_f=\Ric+\Hess_f$ is proportional to the commutator of $D_f$ and $\nabla$: for any vector field $X$ and spinor $\psi$, the following weighted Ricci identity holds
\begin{equation}\label{eqn: weighted Ricci identity}
    [D_f,\nabla_X]\psi=\frac{1}{2}\Ric_f(X)\cdot\psi.
\end{equation}
Finally, the weighted Dirac operator is unitarily equivalent to the standard Dirac operator. 
Indeed, a routine calculation shows that for every spinor $\psi$,
\begin{equation}\label{eqn: unitary equivalence of Dirac and weighted Dirac}
    D_f\psi=e^{f/2}D(e^{-f/2}\psi),
\end{equation}
and the map $L^2\to L^2_f$ defined by $\psi\mapsto e^{f/2}\psi$ is unitary. 
The reader is referred to \cite[\S 1]{BO} for proofs of the above statements and other fundamental properties of weighted spin manifolds.
\section{Variational formulas}\label{sec: Variational formulas}

The purpose of this section is to compute the variations of spinorial quantities which are used in the proof of the monotonicity theorems. 
These variational formulas hold on general manifolds with boundary.
Since the spin bundle varies with the metric, the variational formulas derived here are to be understood within the framework of the generalized cylinder construction of Section \ref{sec: Spin geometry of generalized cylinders}. Throughout this section, $(M^n,g,f)$ is a weighted spin manifold, $\psi$ is a spinor on $M$, and
\begin{equation}
    \dot{g}=\partial_tg,
    \qquad \qquad 
    \dot{f}=\partial_tf,
    \qquad \qquad 
    \dot{\psi}=\overline{\nabla}_t\psi,
\end{equation} 
denote variations of $g,f$ and $\psi$. 

The variation of the gradient of a spinor involves two important tensors which are defined below. For any spinor $\psi$ on $M$, define the real symmetric 2-tensor $\langle \nabla \psi\otimes \nabla \psi\rangle $ by
\begin{equation}
    \langle \nabla \psi\otimes \nabla \psi\rangle (X,Y)
    =\Rea\langle \nabla_X\psi,\nabla_Y\psi\rangle.
\end{equation} 
The symmetry of $\langle \nabla \psi\otimes \nabla \psi\rangle$ is a consequence of the fact that the spin metric is Hermitian. Further, define the real 3-tensor $T_{\psi}$ on $M$ by
\begin{equation}\label{eqn: definition of T tensor 1}
    T_{\psi}(X,Y,Z)
    =\frac{1}{2}\Rea\biggl( \langle (X\wedge Y)\cdot \psi,\nabla_Z\psi\rangle+\langle (X\wedge Z)\cdot \psi,\nabla_Y\psi\rangle\biggr).
\end{equation}
By construction $T_{\psi}$ is symmetric in the second and third components; that is, 
\begin{equation}
    T_{\psi}(X,Y,Z)=T_{\psi}(X,Z,Y).
\end{equation}
Consequently, the 2-tensor $\Div( T_{\psi})=(\nabla_kT_{\psi})(e_k,\cdot,\cdot)$ is symmetric. Also, for later use, note that the formula $X\wedge Y=X\cdot Y +\langle X,Y\rangle$ and anti-Hermiticity of Clifford multiplication implies that 
\begin{align}\label{eqn: defintion of T tensor 2}
    T_{\psi}(X,Y,Z)
    =\frac{1}{2}\Rea \langle X\cdot \psi,Y\cdot \nabla_Z\psi+Z\cdot \nabla_Y\psi\rangle 
    -\frac{1}{4}(\langle X,Y\rangle Z+\langle X,Z\rangle Y)|\psi|^2.
\end{align}
A derivation of the following variational formula can be found in \cite{AWW} or \cite[p.\ 65]{S}.

\begin{proposition}
[Variation of $|\nabla \psi|^2$]\label{prop: variation of spinorial gradient squared}
The squared norm of the gradient of a spinor evolves by
\begin{equation}
    \partial_t|\nabla \psi|^2
    =-\langle \dot{g},\langle \nabla \psi\otimes \nabla \psi\rangle \rangle 
        +\frac{1}{2}\langle \nabla \dot{g},T_{\psi}\rangle
        +2\Rea\langle \nabla \dot{\psi},\nabla \psi\rangle.
\end{equation}
\end{proposition}

The previous proposition shows that the variation of the squared norm of the gradient of a spinor depends on the term $\langle \nabla \dot{g},T_{\psi}\rangle$, which, upon integration by parts, can be written as the inner product of $\dot{g}$ with the weighted divergence $\Div_f(T_{\psi})$. 
As is shown below, the weighted divergence of $T_{\psi}$ depends on the Lie derivative of the metric with respect to the vector field $V_{\psi,f}$, defined by 
\begin{equation}
    V_{\psi,f}=\Rea\langle \psi,e_i\cdot D_f\psi\rangle e_i.
\end{equation}

\begin{lemma}\label{lem: Lie derivative}
The Lie derivative of the metric in the direction of $V_{\psi,f}$ is 
\begin{equation}
    (\Ld_{V_{\psi,f}}g)(X,Y)
        =\Rea\left(\langle \psi,X\cdot \nabla_YD_f\psi+Y\cdot \nabla_XD_f\psi\rangle 
        -\langle D_f\psi,X\cdot \nabla_Y\psi+Y\cdot \nabla_X\psi\rangle \right).
\end{equation}
\end{lemma}

\begin{proof}
Choose a local orthonormal frame $e_1,\dots, e_n$ of $TM$ for which $\nabla e_i=0$ at a point $p$. At $p$,
\begin{align}
    (\Ld_{V_{\psi,f}}g)_{ij}
        &=\nabla_iV_j+\nabla_jV_i
        \\
        &=\Rea \biggl(
        \langle \nabla_i\psi,e_j\cdot D_f\psi\rangle 
        +\langle \psi,e_j\cdot\nabla_iD_f\psi\rangle 
        +\langle \nabla_j\psi,e_i\cdot D_f\psi\rangle 
        +\langle \psi,e_i\cdot\nabla_jD_f\psi\rangle 
        \biggr)
        \nonumber\\
        &=\Rea \biggl(
        \langle \psi,e_i\cdot \nabla_jD_f\psi+e_j\cdot \nabla_iD_f\psi\rangle 
        -\langle D_f\psi,e_i\cdot \nabla_j\psi+e_j\cdot \nabla_i\psi\rangle 
        \biggr).
        \nonumber
\end{align}
\end{proof}

The next lemma is crucial for applications to Ricci flow, because it shows that the weighted divergence $\Div_f(T_{\psi})$, appearing in the variation of the weighted Dirichlet energy $\int_M|\nabla \psi|^2\,e^{-f}$, is closely related to the weighted Ricci curvature $\Ric_f$.

\begin{lemma}\label{lem: div_f(T) in terms of Ric}
The 2-tensor $\Div_f(T_{\psi})$ satisfies
\begin{align}
    \Div_f(T_{\psi})(X,Y)
        &=-\frac{1}{2}\Ric_f(X,Y)|\psi|^2 
        +\frac{1}{2}\Hess_{|\psi|^2}(X,Y) 
        -2\langle \nabla_X\psi,\nabla_Y\psi\rangle\\
        &\qquad\qquad \qquad 
        +\Rea \langle D_f\psi,X\cdot \nabla_Y\psi+Y\cdot \nabla_X\psi\rangle
        +\frac{1}{2}\Ld_{V_{\psi,f}}g
        .\nonumber
\end{align}
\end{lemma}

\begin{proof}
For this proof, write $T$ for $T_{\psi}$ to simplify notation. 
In an orthonormal frame, Definition (\ref{eqn: defintion of T tensor 2}) of $T$ implies that
\begin{align}
    T_{ijk}
    &=\frac{1}{2}\Rea \langle e_i\cdot \psi,e_j\cdot \nabla_k\psi+e_k\cdot \nabla_j\psi\rangle
    -\frac{1}{4}\left(\delta_{ij}e_k+\delta_{ik}e_j\right)|\psi|^2. 
\end{align}
Fix a point $p\in M$, and choose a local orthonormal frame of $TM$ with $\nabla e_i=0$ at $p$. Then, at $p$, 
\begin{align}
    \Div(T)_{jk}
    &=\nabla_i T_{ijk} \\
    &=\frac{1}{2}\Rea\left(
    \langle D\psi,e_j\cdot \nabla_k\psi+e_k\cdot \nabla_j\psi\rangle \right.\nonumber \\
    &\qquad \qquad \qquad \left.
    +\langle e_i\cdot \psi,e_j\cdot \nabla_i\nabla_k\psi+e_k\cdot \nabla_i\nabla_j\psi\rangle  
    -\Hess_{|\psi|^2}(e_j,e_k)
    \right)  \nonumber \\
    &=\frac{1}{2}\Rea\left(
    \langle D\psi,e_j\cdot \nabla_k\psi+e_k\cdot \nabla_j\psi\rangle \right.\nonumber \\
    &\qquad \qquad \qquad \left.
    -\langle \psi,e_i\cdot e_j\cdot \nabla_i\nabla_k\psi+e_i\cdot e_k\cdot \nabla_i\nabla_j\psi\rangle  
    -\Hess_{|\psi|^2}(e_j,e_k)
    \right)  \nonumber \\
    &=\frac{1}{2}\Rea\left(
    \langle D\psi,e_j\cdot \nabla_k\psi+e_k\cdot \nabla_j\psi\rangle 
    +\langle \psi,e_j\cdot D\nabla_k\psi+e_k\cdot D\nabla_j\psi\rangle \right.\nonumber \\
    &\qquad \qquad \qquad \left.
    +2\langle \psi, \nabla_j\nabla_k\psi+\nabla_k\nabla_j\psi\rangle  -\Hess_{|\psi|^2}(e_j,e_k)
    \right) .
    \nonumber
\end{align}
Recall that 
$
    \Div_f(T)=\Div(T)-T(\nabla f,\cdot,\cdot).
$
Then note that (\ref{eqn: defintion of T tensor 2}) implies
\begin{align}\label{eqn: T tensor contracted with grad f}
    T(\nabla f,e_j,e_k)
    &=\frac{1}{2}\Rea \langle (\nabla f)\cdot \psi,e_j\cdot \nabla_k\psi+e_k\cdot \nabla_j\psi\rangle \\
    &\qquad\qquad\qquad
    -\frac{1}{4}\left((\nabla_jf)e_k+(\nabla_kf)e_j\right)|\psi|^2. \nonumber\\
    &=\frac{1}{2}\Rea \left( \frac{1}{2}\langle (\nabla f)\cdot \psi,e_j\cdot \nabla_k\psi+e_k\cdot \nabla_j\psi\rangle
    \right.\nonumber\\
    &\qquad \qquad \qquad \left. 
    -\frac{1}{2}\langle \psi,(\nabla f)\cdot e_j\cdot \nabla_k\psi+(\nabla f)\cdot e_k\cdot \nabla_j\psi\rangle\right)\nonumber\\
    &\qquad \qquad \qquad\qquad \qquad \qquad
    -\frac{1}{4}\left((\nabla_jf)e_k+(\nabla_kf)e_j\right)|\psi|^2. \nonumber\\
    &=\frac{1}{2}\Rea \left( \frac{1}{2}\langle (\nabla f)\cdot \psi,e_j\cdot \nabla_k\psi+e_k\cdot \nabla_j\psi\rangle
    \right.\nonumber\\
    &\qquad \qquad \qquad \left. 
    +\frac{1}{2}\langle \psi,e_j\cdot (\nabla f)\cdot \nabla_k\psi+e_k\cdot (\nabla f)\cdot \nabla_j\psi\rangle\right). \nonumber
\end{align}
Combining the last equation with the one for $\Div(T)_{jk}$ above and using the definition of the weighted Dirac operator implies
\begin{align}
    \Div_f(T)_{jk}
    &=\frac{1}{2}\Rea\left(
    \langle D_f\psi,e_j\cdot \nabla_k\psi+e_k\cdot \nabla_j\psi\rangle 
    +\langle \psi,e_j\cdot D_f\nabla_k\psi+e_k\cdot D_f\nabla_j\psi\rangle \right.\nonumber \\
    &\qquad \qquad \qquad \left.
    +2\langle \psi, \nabla_j\nabla_k\psi+\nabla_k\nabla_j\psi\rangle  -\Hess_{|\psi|^2}(e_j,e_k)
    \right) .
\end{align}
When the weighted Ricci identity (\ref{eqn: weighted Ricci identity})
is applied to the second term above, and the third term is rewritten using the symmetry of the Hessian,
\begin{equation}
    \Hess_{|\psi|^2}(e_j,e_k)=2\Rea \left(\langle \psi,\nabla_j\nabla_k\psi\rangle +\langle \nabla_j\psi,\nabla_k\psi\rangle\right),
\end{equation}
and the symmetry of $\Rea\langle \cdot,\cdot\rangle$, it follows that
\begin{align}\label{eqn: intermediate div(T) calculation}
    \Div_f(T)_{jk}
    &=\frac{1}{2}\Rea\left(
    \langle D_f\psi,e_j\cdot \nabla_k\psi+e_k\cdot \nabla_j\psi\rangle 
    +\langle \psi,e_j\cdot \nabla_kD_f\psi+e_k\cdot \nabla_jD_f\psi\rangle \right. \\
    &\qquad \qquad \qquad \left.
    +\frac{1}{2}\langle \psi,(\Ric_f)_{kl}e_j\cdot e_l\cdot \psi +(\Ric_f)_{jl}e_k\cdot e_l\cdot \psi\rangle \right. \nonumber \\
    &\qquad \qquad \qquad \left.
    -4\langle \nabla_j\psi,\nabla_k\psi\rangle   +\Hess_{|\psi|^2}(e_j,e_k)
    \right) .\nonumber \\
    &=\frac{1}{2}\Rea\left(
    \langle D_f\psi,e_j\cdot \nabla_k\psi+e_k\cdot \nabla_j\psi\rangle 
    +\langle \psi,e_j\cdot \nabla_kD_f\psi+e_k\cdot \nabla_jD_f\psi\rangle \right.\nonumber \\
    &\qquad \qquad \qquad \left.
    -(\Ric_f)_{jk}|\psi|^2
    -4\langle \nabla_j\psi,\nabla_k\psi\rangle   +\Hess_{|\psi|^2}(e_j,e_k)
    \right) .\nonumber 
\end{align}
Using Lemma \ref{lem: Lie derivative}, the second term above can be rewritten in terms of the first and $\Ld_{V_{\psi,f}}g$, as claimed.
\end{proof}

\begin{proposition}\label{prop: variation of weighted Dirichlet energy}
The variation of the weighted spinorial Dirichlet energy is
\begin{align}
    \frac{d}{dt}&\int_M|\nabla \psi|^2\,e^{-f}dV \\
    &=\int_M\left(
        \left(\frac{1}{2}\tr(\dot{g})-\dot{f}\right)|\nabla \psi|^2
        -2\Rea\langle \dot{\psi},\Delta_f\psi\rangle
        -\left\langle \dot{g},\frac{1}{2}\Div_f(T_{\psi})+\langle \nabla \psi\otimes \nabla \psi\rangle\right\rangle
        \right)e^{-f}dV 
        \nonumber\\ &\qquad\qquad\qquad
        +\int_{\partial M}\left(
        \frac{1}{2}\langle \dot{g},T_{\psi}(\nu,\cdot)\rangle 
        +2\Rea\langle \dot{\psi},\nabla_{\nu}\psi\rangle\right)e^{-f}dA.
        \nonumber
\end{align}
\end{proposition}

\begin{proof}
The proof consists of computing the derivative
\begin{equation}
    \partial_t\left(|\nabla \psi|^2\,e^{-f}dV\right)
        =\partial_t(|\nabla \psi|^2)\,e^{-f}dV+|\nabla \psi|^2\partial_t(e^{-f}dV).
\end{equation}
The second term depends on the variation of the weighted measure. Recall the variational formula $\partial_t(dV)=\frac{1}{2}\tr(\dot{g})dV$, 
which follows from Jacobi's formula: $\frac{d}{dt}\det(A)=\det(A)\tr(A^{-1}\frac{dA}{dt})$ for any invertible square matrix $A$. Hence 
\begin{equation}
    \partial_t(e^{-f}dV)
        =\left( \frac{1}{2}\tr(\dot{g})-\dot{f}\right)e^{-f}dV .
\end{equation}
The variation of $|\nabla \psi|^2$ was computed in Proposition \ref{prop: variation of spinorial gradient squared}:
\begin{equation}
    \partial_t|\nabla \psi|^2
    =-\langle \dot{g},\langle \nabla \psi\otimes \nabla \psi\rangle \rangle 
        +\frac{1}{2}\langle \nabla \dot{g},T_{\psi}\rangle
        +2\Rea\langle \nabla \dot{\psi},\nabla \psi\rangle.
\end{equation}
The proposition now follows from the weighted divergence theorem; see Appendix \ref{sec: weighted integration by parts}.
\end{proof}

The derivatives in the variational formula to follow are arranged for ease of reference in the proofs of the monotonicity formulas in Section \ref{sec: Monotonicity}.

\begin{proposition}\label{prop: variation of Scal_f part}
The following variational formula holds:
\begin{align}
    \frac{d}{dt}\int_M\Scal_f|\psi|^2e^{-f}dV
        &=\int_M\left(-\langle \dot{g},\Ric_f|\psi|^2\rangle
        +\Div_f^2(\dot{g})|\psi|^2
        +2\Rea \langle \dot{\psi},\Scal_f\psi\rangle
        \right. \nonumber \\ &\qquad \quad \quad \left.
        +4\left(\frac{1}{2}\tr(\dot{g})-\dot{f}\right)(\Rea\langle D_f^2\psi,\psi\rangle-|\nabla\psi|^2)
        \right)e^{-f}dV
        \\ &\qquad \quad \quad 
        +2\int_{\partial M}\left(
        \left(\frac{1}{2}\tr(\dot{g})-\dot{f}\right)\nabla_{\nu}|\psi|^2
        -|\psi|^2\nabla_{\nu}\left(\frac{1}{2}\tr(\dot{g})-\dot{f}\right)
        \right)e^{-f}dA. \nonumber
\end{align}
\end{proposition}

\begin{proof}
Recall the variational formulas
\begin{align}
    \partial_t\Scal 
        &=\Div^2(\dot{g})-\Delta\tr(\dot{g})-\langle \dot{g},\Ric\rangle \\
    \partial_t(dV)
        &=\frac{1}{2}\tr(\dot{g})dV \\
    \partial_t\Delta f
        &=\Delta \dot{f}-\left\langle \Div(\dot{g})-\frac{1}{2}\nabla\tr(\dot{g}),\nabla f\right\rangle-\langle \dot{g},\Hess_f\rangle \\
    \partial_t|\nabla f|^2
        &=2\langle \nabla \dot{f},\nabla f\rangle-\langle \dot{g},\nabla f\otimes \nabla f\rangle.
\end{align}
Rewriting $\Div$ in terms of $\Div_f$ and combining the above equations implies
\begin{align}\label{eqn: variation of Scal_f}
    \partial_t\Scal_f 
        &=\Div_f^2(\dot{g})-2\Delta_f\left(\frac{1}{2}\tr(\dot{g})-\dot{f}\right)-\langle \dot{g},\Ric_f\rangle \\
    \partial_t|\psi|^2
        &=2\Rea \langle \dot{\psi},\psi\rangle \\
    \partial_t(e^{-f}dV)
        &=\left(\frac{1}{2}\tr(\dot{g})-\dot{f}\right)e^{-f}dV.
\end{align}
Combined, these imply
\begin{align}
    \partial_t(\Scal_f|\psi|^2e^{-f}dV)
        &=\left(\Div_f^2(\dot{g})-2\Delta_f\left(\frac{1}{2}\tr(\dot{g})-\dot{f}\right)-\langle \dot{g},\Ric_f\rangle\right)|\psi|^2\,e^{-f}dV \\
        &\qquad \qquad \qquad 
        +\left(2\Rea \langle \dot{\psi},\Scal_f\psi\rangle +\left(\frac{1}{2}\tr(\dot{g})-\dot{f}\right)\Scal_f|\psi|^2\right)e^{-f}dV. \nonumber
\end{align}
Integration by parts implies
\begin{align}
    \frac{d}{dt}\int_M\Scal_f|\psi|^2e^{-f}dV
        &=\int_M\left(-\langle \dot{g},\Ric_f|\psi|^2\rangle
        +\Div_f^2(\dot{g})|\psi|^2
        +2\Rea \langle \dot{\psi},\Scal_f\psi\rangle
        \right. \nonumber\\ &\qquad \quad \quad \left.
        +\left(\frac{1}{2}\tr(\dot{g})-\dot{f}\right)(\Scal_f|\psi|^2-2\Delta_f|\psi|^2)
        \right)e^{-f}dV 
        \nonumber \\ &\qquad \quad \quad 
        +2\int_{\partial M}\left(
        \left(\frac{1}{2}\tr(\dot{g})-\dot{f}\right)\nabla_{\nu}|\psi|^2
        -|\psi|^2\nabla_{\nu}\left(\frac{1}{2}\tr(\dot{g})-\dot{f}\right)
        \right)e^{-f}dA 
        \nonumber \\
        &=\int_M\left(-\langle \dot{g},\Ric_f|\psi|^2\rangle
        +\Div_f^2(\dot{g})|\psi|^2
        +2\Rea \langle \dot{\psi},\Scal_f\psi\rangle
        \right. \\ &\qquad \quad \quad \left.
        +4\left(\frac{1}{2}\tr(\dot{g})-\dot{f}\right)(\Rea\langle D_f^2\psi,\psi\rangle-|\nabla\psi|^2)
        \right)e^{-f}dV
        \nonumber\\ &\qquad \quad \quad 
        +2\int_{\partial M}\left(
        \left(\frac{1}{2}\tr(\dot{g})-\dot{f}\right)\nabla_{\nu}|\psi|^2
        -|\psi|^2\nabla_{\nu}\left(\frac{1}{2}\tr(\dot{g})-\dot{f}\right)
        \right)e^{-f}dA. 
        \nonumber
\end{align}
where the last equality has used the weighted Bochner formula
\begin{equation}
    \Delta_f|\psi|^2=-2\Rea\langle D_f^2\psi,\psi\rangle +\frac{1}{2}\Scal_f|\psi|^2+2|\nabla \psi|^2,
\end{equation}
which follows easily from the weighted Lichnerowicz formula (\ref{eqn: weighted Lichnerowicz}).
\end{proof}

\begin{corollary}
The following variational formula holds:
\begin{align}\label{eqn: derivative of weighted Dirichlet plus Scal_f term, general case}
    \frac{d}{dt}\int_M&\left( 4|\nabla \psi|^2+\Scal_f|\psi|^2\right)e^{-f}dV
        \\ 
        &=\int_M\biggl(
        \left(\frac{1}{2}\tr(\dot{g})-\dot{f}\right)4\Rea\langle D_f^2\psi,\psi\rangle
        +8\Rea \langle \dot{\psi},D^2_f\psi\rangle 
        \nonumber \\ &\qquad \qquad\quad 
        -\langle \dot{g},2\Div_f(T_{\psi})+4\langle \nabla \psi\otimes \nabla \psi\rangle +\Ric_f|\psi|^2\rangle 
        +\Div_f^2(\dot{g})|\psi|^2
        \biggr) e^{-f}dV
        \nonumber\\&\qquad\qquad
        +\int_{\partial M}\biggl(
        2\langle \dot{g},T_{\psi}(\nu,\cdot)\rangle 
        +8\Rea\langle \dot{\psi},\nabla_{\nu}\psi\rangle
        +2\left(\frac{1}{2}\tr(\dot{g})-\dot{f}\right)\nabla_{\nu}|\psi|^2
        \nonumber \\ &\qquad \qquad\qquad \qquad\quad
        -2|\psi|^2\nabla_{\nu}\left(\frac{1}{2}\tr(\dot{g})-\dot{f}\right)
        \biggr)e^{-f}dA.
        \nonumber
\end{align}
In particular, if $D_f\psi=0$, then
\begin{align}\label{eqn: derivative of weighted Dirichlet plus Scal_f term}
    \frac{d}{dt}\int_M&\left( 4|\nabla \psi|^2+\Scal_f|\psi|^2\right)e^{-f}dV
        \nonumber\\ 
        &=\int_{\partial M}\biggl(
        2\langle \dot{g},T_{\psi}(\nu,\cdot)\rangle 
        +8\Rea\langle \dot{\psi},\nabla_{\nu}\psi\rangle
        +2\left(\frac{1}{2}\tr(\dot{g})-\dot{f}\right)\nabla_{\nu}|\psi|^2
        \\&\qquad\qquad
        -\langle \dot{g}(\nu,\cdot),\nabla |\psi|^2\rangle
        +|\psi|^2\biggl\langle\Div_f(\dot{g})-2\nabla\left(\frac{1}{2}\tr(\dot{g})-\dot{f}\right),\nu\biggr\rangle
        \biggr)e^{-f}dA.
        \nonumber
\end{align}
\end{corollary}

\begin{proof}
The first part of the corollary is immediate from the combination of Propositions \ref{prop: variation of weighted Dirichlet energy} and \ref{prop: variation of Scal_f part}. To prove the second part, note that if $D_f\psi=0$, then Lemma \ref{lem: div_f(T) in terms of Ric} implies that 
\begin{equation}
    2\Div_f(T_{\psi})+4\langle \nabla \psi\otimes \nabla \psi\rangle+\Ric_f|\psi|^2=\Hess_{|\psi|^2}.
\end{equation}
Furthermore, the weighted divergence theorem (see Appendix \ref{sec: weighted integration by parts}), applied twice, implies
\begin{align}
    \int_M\Div_f^2(\dot{g})|\psi|^2\,e^{-f}dV
    &=\int_M\langle \dot{g},\Hess_{|\psi|^2}\rangle \,e^{-f}dV
    \\&\qquad \qquad  
    +\int_{\partial M}\biggl(\langle \Div_f(\dot{g}),\nu\rangle |\psi|^2
    -\langle \dot{g}(\nu,\cdot),\nabla |\psi|^2\rangle 
    \biggr)e^{-f}dA.
    \nonumber
\end{align}
The first part of the corollary combined with the latter two formulas implies the second part of the corollary.
In particular, the $\Hess_{|\psi|^2}$ terms coming from $\Div_f(T_{\psi})$ (the variation of $|\nabla\psi|^2$) and from $\Div_f^2(\dot{g})$ (the variation of $\Scal_f$) cancel.
\end{proof}

For later use, the formula for the variation of the Dirac operator is recorded below. A derivation of this formula can be found in \cite[Thm.\ 5.1]{BGM}, for example.

\begin{lemma}
[Variation of Dirac operator]
\label{lem: variation of Dirac operator}
The Dirac operator evolves by
\begin{align}
    \nablaCyl_{\partial_t} D\psi
    &=D\dot{\psi} -\frac{1}{2}\dot{g}(e_i)\cdot \nabla_i\psi
    +\frac{1}{4}\left(\nabla \tr(\dot{g})-\Div(\dot{g})\right)\cdot \psi.
\end{align}
\end{lemma}
\section{The energy functional}\label{sec: Monotonicity}

This section applies the variational formulas derived in Section \ref{sec: Variational formulas} to the special case of asymptotically Euclidean manifolds to prove existence and uniqueness of critical points of the energy (Theorem \ref{thm: existence and uniqueness of critical points, intro}), and the monotonicity theorem (Theorem \ref{thm: monotonicity intro}).

A Riemannian spin manifold $(M^n,g)$ is called asymptotically Euclidean (AE) of order $\tau$ if there exists a compact subset $K\subset M$, a radius $\rho>0$, and a diffeomorphism $\Phi:M\setminus K\to \R^n\setminus B_{\rho}(0)$, with respect to which, for all $k\in \N$,
\begin{equation}
    g_{ij}=\delta_{ij}+O(r^{-\tau}), \qquad \partial^kg_{ij}=O(r^{-\tau-k}),
\end{equation}
for any partial derivative of order $k$ as $r\to \infty$, where $r=|\Phi|$ is the Euclidean distance function. 
The set $M\setminus K$ is called the end of $M$. (The results of this section extend in a straightforward manner to AE manifolds with multiple ends, though they are not pursued here.)
The AE structure $\Phi$ defines a trivialization of the spin bundle at infinity. A spinor $\psi$ defined on the end of $M$ is called \emph{constant} (with respect to the asymptotic coordinates) if $\psi=(\Phi^{-1})^*\psi_0$, for some constant spinor $\psi_0$ on $\R^n$.
In what follows, denote by $S_{\rho}=r^{-1}(\rho)\subset M$ the coordinate sphere of radius $\rho$.

The appropriate analytic tools for studying AE manifolds are the \emph{weighted H\"older spaces} $C^{k,\alpha}_{\beta}(M)$, whose precise definitions are stated in Appendix \ref{sec: appendix}.
These spaces share many of the global elliptic regularity results which hold for the usual H\"older spaces on compact manifolds.
The index $\beta$ is important because it denotes the {\it order of growth}: functions in $C^{k,\alpha}_{\beta}(M)$ grow at most like $r^{\beta}$.
In particular, if the metric $g$ is AE of order $\tau$ on $M=\R^n$, then in the AE coordinate system, $g-\delta$ lies in $C^{k,\alpha}_{-\tau}(M)$ for all $k\in \N$ and the scalar curvature of $g$ lies in $C^{k,\alpha}_{-\tau-2}(M)$ for all $k\in \N$.

\subsection{Critical points.}\label{sec: critical points}
Let $(M^n,g)$ be a spin, AE manifold of order $\tau>\frac{n-2}{2}$. Fix a smooth spinor $\psi_0$ which is constant at infinity with respect to the AE coordinate system and with $|\psi_0|\to 1$ at infinity. Define the energy functional
\begin{align}
    \mathscr{E}_g(\psi,f)=\int_M\Big(4|\nabla \psi|^2+\Scal_f(|\psi|^2-1)\Big)e^{-f}dV_g,
\end{align}
on the space of spinors $\psi$ such that $\psi-\psi_0\in C^{2,\alpha}_{-\tau}(M)$ and the space of functions $f\in C^{2,\alpha}_{-\tau}(M)$. 
(These boundary conditions extend in a straightforward manner to other non-compact geometries.)

The energy generalizes various well-known functionals. 
If the spinor is zero, the energy equals minus Perelman's weighted Hilbert-Einstein action;
if the spinor has unit norm and the weight is zero, the energy is the ``spinorial energy'' introduced for closed manifolds in \cite{AWW};
if the weighted scalar curvature vanishes, then the energy equals the weighted Dirichlet energy of the spinor.
Furthermore, if the weight is zero, spinors minimizing (\ref{eqn: energy functional}) are precisely the Witten spinors, and the value of the energy equals the difference between the ADM mass and the Hilbert-Einstein action, also known as the Regge-Teitelboim Hamiltonian; 
for general $f$, the spinors minimizing (\ref{eqn: energy functional}) are precisely the weighted Witten spinors, and the value of the energy equals the difference between the weighted ADM mass and Perelman's weighted Hilbert-Einstein action \cite{BO,DO,DO2}. 

\begin{proposition}
[Variation of $\mathscr{E}$]
\label{prop: variation of E}
The variation of $\mathscr{E}_g$ in the directions $\dot{\psi},\dot{f}\in C^{2,\alpha}_{-\tau}(M)$ is
\begin{align}\label{eqn: variation of E}
    \frac{d}{dt}\mathscr{E}_g(\psi,f)
        &=\int_M\biggl(\dot{f}\left(\Scal_f-4\Rea\langle D_f^2\psi,\psi\rangle\right)
        +8\Rea \langle \dot{\psi},D^2_f\psi\rangle
        \biggr) e^{-f}dV.
\end{align}
In particular, the pair $(\psi,f)$ is critical for $\mathscr{E}_g$ if and only if 
\begin{equation}
    \Scal_f=0 \qquad \text{and} \qquad D_f\psi=0.
\end{equation}
\end{proposition}

\begin{proof}
Using (\ref{eqn: derivative of weighted Dirichlet plus Scal_f term, general case}), it remains to compute the variation of $\int_M\Scal_f\,e^{-f}dV$ with respect to $\dot{f}$. This is achieved via (\ref{eqn: variation of Scal_f}): when $\dot{g}=0$, it follows that
\begin{equation}
    \partial_t\left(\Scal_f\,e^{-f}\right)=\left(2\Delta_f\dot{f}-\dot{f}\Scal_f\right)\,e^{-f}.
\end{equation}
Hence, the weighted divergence theorem implies
\begin{align}
    \frac{d}{dt}\int_M\Scal_f\,e^{-f}dV
        &=\int_M\left(2\Delta_f\dot{f}-\dot{f}\Scal_f\right)\,e^{-f}dV
        \\
        &=-\int_M\dot{f}\Scal_f\,e^{-f}dV+\lim_{\rho\to \infty}\int_{S_{\rho}}2\nabla_{\nu}\dot{f} \,e^{-f}dA.
        \nonumber
\end{align}
Combined with the $\dot{g}=0$ version of (\ref{eqn: derivative of weighted Dirichlet plus Scal_f term, general case}), it follows that
\begin{align}
    \frac{d}{dt}\mathscr{E}_g(\psi,f)
        &=\int_M\biggl(\dot{f}\left(\Scal_f-4\Rea\langle D_f^2\psi,\psi\rangle\right)
        +8\Rea \langle \dot{\psi},D^2_f\psi\rangle
        \biggr) e^{-f}dV
        \\&\qquad\qquad
        +\lim_{\rho\to \infty}\int_{S_{\rho}}\biggl(
        8\Rea\langle \dot{\psi},\nabla_{\nu}\psi\rangle
        -2\dot{f}\nabla_{\nu}|\psi|^2
        +2(|\psi|^2-1)\nabla_{\nu}\dot{f}
        \biggr)e^{-f}dA.
        \nonumber
\end{align}
Since $\dot{\psi}$, $\dot{f}$, and $|\psi|^2-1$ all lie in $C^{2,\alpha}_{-\tau}$, the boundary terms vanish when $\tau>\frac{n-2}{2}$.

It follows immediately from (\ref{eqn: variation of E}) and the fundamental lemma of the calculus of variations that the pair $(\psi,f)$ is critical for $\mathscr{E}_g$ if and only if $D_f^2\psi=0$ and $\Scal_f=0$. 
It therefore remains to show that if $D_f^2\psi=0$ and $\Scal_f=0$, then in fact $D_f\psi=0$. 
If $D_f^2\psi=0$, then the spinor $\phi:=D_f\psi$ lies in $C^{1,\alpha}_{-\tau-1}$ and satisfies $D_f\phi=0$.
Applying the weighted Lichnerowicz formula with the assumption $\Scal_f=0$, and integrating by parts (the boundary term vanishes because $\tau > (n-2)/2$), implies
\begin{align*}
    0
    &=\int_M\langle D^2_f\phi,\phi\rangle \,e^{-f} dV_g 
    =-\int_M\langle \dL_f\phi,\phi\rangle\,e^{-f} dV_g 
    =\int_M |\nabla \phi|^2\,e^{-f} dV_g.
\end{align*}
Hence $\nabla \phi=0$, so $|\phi|^2$ is a constant, which must be zero since $\phi$ vanishes at infinity. Thus $D_f\psi=0$.
\end{proof}

\begin{theorem}
[Existence and uniqueness of critical points; Theorem \ref{thm: existence and uniqueness of critical points, intro} restated]
\label{thm: existence and uniqueness of critical points}
If $(M^n,g)$ has nonnegative scalar curvature, there exists a unique critical point $(\psi_g,f_g)$ of $\mathscr{E}_g$ such that $\psi_g-\psi_0$ and $f_g$ lie in $C^{2,\alpha}_{-\tau}(M)$. Moreover, $(\psi_g,f_g)$ is a min-max critical point,
\begin{equation}
    \mathscr{E}_g(\psi_g,f_g)=\max_{f} \min_{\psi}\mathscr{E}_g(\psi,f),
\end{equation}
where the min-max is taken over all $\psi$ such that $\psi-\psi_0\in C^{2,\alpha}_{-\tau}(M)$ and all $f\in C^{2,\alpha}_{-\tau}(M)$.
\end{theorem}

\begin{proof}
The proof proceeds in two steps. Step 1 shows that given any $f\in C^{2,\alpha}_{-\tau}$, there exists a unique a $D_f$-harmonic spinor $\psi_f$ which is asymptotic to $\psi_0$ and that $\psi_f$ globally minimizes $\mathscr{E}_g(\cdot,f)$ over all spinors asymptotic to $\psi_0$. 
Step 2 shows that there exists a unique $f_g\in C^{2,\alpha}_{-\tau}$ which solves $\Scal_{f_g}=0$ and that $f_g$ globally maximizes $\mathscr{E}_g(\psi_{f},f)$ over all $f\in C^{2,\alpha}_{-\tau}$, with $\psi_{f}$ given by Step 1.
The pair $(\psi_{f_g},f_g)$ is then the desired critical point of $\mathscr{E}_g$.
\\\\
\noindent
\emph{Claim 1:} Given any $f\in C^{2,\alpha}_{-\tau}$, there exists a unique $D_f$-harmonic spinor $\psi_f$ with $\psi_f-\psi_0\in C^{2,\alpha}_{-\tau}$, and $\psi_f$ globally minimizes $\mathscr{E}_g(\cdot,f)$ over all spinors asymptotic to $\psi_0$ in $C^{2,\alpha}_{-\tau}$.

\emph{Proof of Claim 1.}
By the proof of Witten's positive mass theorem, there exists a unique smooth spinor $\psi$ on $M$ such that $\psi-\psi_0\in C^{2,\alpha}_{-\tau}$ and $D\psi=0$. 
Choose any $f\in C^{2,\alpha}_{-\tau}$. 
Then by the unitary equivalence (\ref{eqn: unitary equivalence of Dirac and weighted Dirac}), the spinor $\psi_f=e^{f/2}\psi$ solves $D_f\psi_f=0$. 

It remains to show that $\psi_f$ minimizes $\mathscr{E}_g(\cdot,f)$. 
This is achieved by showing that
\begin{equation}\label{eqn: intermediate minimization eqn 0}
    \mathscr{E}_g(\psi_f+\phi,f)\geq \mathscr{E}_g(\psi_f,f),
\end{equation}
for all compactly supported smooth spinors $\phi$; density of $C^{\infty}_c(M)$ in $C^{2,\alpha}_{-\tau}(M)$ then concludes the proof. 
Let $\phi$ be a compactly supported smooth spinor on $M$. Integration by parts (the boundary term vanishes because $\phi$ is compactly supported), the weighted Lichnerowicz formula (\ref{eqn: weighted Lichnerowicz}) and the assumption $D_f\psi_f=0$ imply
\begin{align}
    \mathscr{E}_g(\psi_f+\phi,f)-\mathscr{E}_g(\psi_f,f)
    &=4\int_M\left(|\nabla(\psi_f+\phi)|^2+\frac{1}{4}\Scal_f|\psi_f+\phi|^2\right) e^{-f}dV
    \nonumber\\&\qquad \qquad \qquad
    -4\int_M\left(|\nabla \psi_f|^2+\frac{1}{4}\Scal_f|\psi_f|^2\right)e^{-f}dV
    \nonumber\\
    &=4\int_M\left(|\nabla \phi|^2+2\Rea \langle \nabla \psi_f,\nabla \phi\rangle +\frac{1}{4}\Scal_f|\phi|^2+\frac{1}{2}\Scal_f\Rea \langle \psi_f,\phi\rangle \right)e^{-f}dV
    \nonumber\\
    &=4\int_M\left(|\nabla \phi|^2+\frac{1}{4}\Scal_f|\phi|^2
    +2\Rea \langle D_f^2\psi_f, \phi\rangle  \right)e^{-f}dV
    \nonumber\\
    &=4\int_M\left(|\nabla \phi|^2+\frac{1}{4}\Scal_f|\phi|^2 \right)e^{-f}dV.
    \label{eqn: intermiate minimization eqn}
\end{align}
Below it is shown that the last integral is nonnegative; when $\Scal_f\geq 0$, this immediate. Using the definitions of the weighted Laplacian (\ref{eqn: defn of weighted Laplacian}) and the weighted scalar curvature (\ref{eqn: defn of weighted scalar curvature}), write $\Scal_f$ as
\begin{equation}
    \Scal_f=\Scal+2\Delta_f f+|\nabla f|^2.
\end{equation}
Then integrating (\ref{eqn: intermiate minimization eqn}) by parts (the boundary term vanishes because $\phi$ has compact support) implies
\begin{equation}\label{eqn: intermiate minimization eqn 2}
    \mathscr{E}_g(\psi_f+\phi,f)-\mathscr{E}_g(\psi_f,f)
    =4\int_M\left(|\nabla \phi|^2+\frac{1}{4}\Scal|\phi|^2-\frac{1}{2}\langle \nabla f,\nabla |\phi|^2\rangle +\frac{1}{4}|\nabla f|^2|\phi|^2
    \right)e^{-f}dV.
\end{equation}
The Cauchy-Schwarz inequality combined with Kato's inequality and the Peter-Paul inequality $ab\leq \frac{1}{2\epsilon }a^2+\frac{\epsilon}{2}b^2$ implies, for all $\epsilon>0$,
\begin{equation}
    -\frac{1}{2}\langle \nabla f,\nabla |\phi|^2\rangle 
    \geq -|\nabla f||\phi||\nabla \phi|
    \geq -\frac{1}{2\epsilon}|\nabla f|^2|\phi|^2-\frac{\epsilon}{2}|\nabla \phi|^2.
\end{equation}
Applied with $\epsilon=2$, this implies
\begin{equation}
    |\nabla \phi|^2+\frac{1}{4}\Scal|\phi|^2-\frac{1}{2}\langle \nabla f,\nabla |\phi|^2\rangle +\frac{1}{4}|\nabla f|^2|\phi|^2
    \geq \frac{1}{4}\Scal|\phi|^2 \geq 0.
\end{equation}
This shows that the integrand in (\ref{eqn: intermiate minimization eqn 2}) is nonnegative, proving (\ref{eqn: intermediate minimization eqn 0}), and hence Claim 1.
\hfill\qedsymbol
\\\\
\noindent
\emph{Claim 2:} There exists a unique $f_g\in C^{2,\alpha}_{-\tau}$ which solves $\Scal_{f_g}=0$ and $f_g$ globally maximizes $\mathscr{E}_g(\psi_{f},f)$ over all $f\in C^{2,\alpha}_{-\tau}$, with $\psi_{f}$ given by Claim 1.

\emph{Proof of Claim 2.}
Theorem 2.17 of \cite{BO} proves that there exists a unique $f_g\in C^{2,\alpha}_{-\tau}$ solving $\Scal_f=0$. It remains to show that $f_g$ maximizes $\mathscr{E}_g(\psi_{f},f)$ over all $f\in C^{2,\alpha}_{-\tau}$, with $\psi_{f}$ given by Claim 1.
This is achieved by showing that, for any compactly supported smooth function $h$ on $M$,
\begin{equation}\label{eqn: intermediate minimization 3}
    \mathscr{E}_g(\psi_{f_g+h},f_g+h)\leq \mathscr{E}_g(\psi_{f_g},f_g).
\end{equation}
The density of $C^{\infty}_c(M)$ in $C^{2,\alpha}_{-\tau}(M)$ then concludes the proof. For ease of notation, let $f=f_g$ for the remainder of this proof. 

Let $\psi$ be the unique Witten spinor asymptotic to $\psi_0$; that is, solving $D\psi=0$ and $\psi-\psi_0\in C^{2,\alpha}_{-\tau}$. By the unitary equivalence (\ref{eqn: unitary equivalence of Dirac and weighted Dirac}), $\psi_f=e^{f/2}\psi$ and $\psi_{f+h}=e^{h/2}\psi_f$. Since $\Scal_f=0$, the definition of weighted scalar curvature (\ref{eqn: defn of weighted scalar curvature}) implies
\begin{equation}
    \Scal_{f+h}
    =\Scal_f+2\Delta_fh-|\nabla h|^2
    =2\Delta_fh-|\nabla h|^2.
\end{equation}
Using this, combined with the fact that $\psi_{f+h}=e^{h/2}\psi_f$ implies
\begin{align}
    \mathscr{E}_g(\psi_{f+h},f+h)
    &=\int_M\left(4|\nabla \psi_{f+h}|^2+\Scal_{f+h}(|\psi_{f+h}|^2-1)\right)e^{-f-h}dV
    \\
    &=\int_M\left(4|\nabla e^{h/2}\psi_f|^2
    +(2\Delta_fh-|\nabla h|^2)(|e^{h/2}\psi_f|^2-1)\right)e^{-f-h}dV
    \nonumber\\
    &=\int_M\biggl(4|\nabla \psi_f|^2
    +|\nabla h|^2|\psi_f|^2
    +4\Rea \langle \nabla_{\nabla h}\psi_f,\psi_f\rangle 
    \nonumber\\&\qquad \qquad 
    +(2\Delta_fh)|\psi_f|^2-2(\Delta_fh)e^{-h}-|\nabla h|^2|\psi_f|^2+|\nabla h|^2e^{-h}\biggr)e^{-f}dV.
    \nonumber
\end{align}
Integrating the $\Delta_fh$ terms by parts with respect to the measure $e^{-f}dV$ (the boundary terms vanish because $h$ is compactly supported), implies
\begin{align}
    \mathscr{E}_g(\psi_{f+h},f+h)
    &=\int_M\biggl(4|\nabla \psi_f|^2
    +4\Rea \langle \nabla_{\nabla h}\psi_f,\psi_f\rangle 
    \\&\qquad \qquad 
    -4\Rea \langle \nabla_{\nabla h}\psi_f,\psi_f\rangle 
    +2\langle \nabla h,\nabla e^{-h}\rangle
    +|\nabla h|^2e^{-h}\biggr)e^{-f}dV
    \nonumber\\
    &=\int_M\biggl(4|\nabla \psi_f|^2
    -|\nabla h|^2e^{-h}\biggr)e^{-f}dV
    \nonumber\\
    &=\mathscr{E}_g(\psi_f,f)-\int_M|\nabla h|^2\,e^{-f-h}dV.
    \nonumber
\end{align}
Since the last term on the RHS above is nonpositive, this proves 
(\ref{eqn: intermediate minimization 3}), and hence Claim 2.
\end{proof}

\subsection{Ricci flow monotonicity.} 

An \emph{asymptotically Euclidean Ricci flow} is defined to be a Ricci flow starting at an AE manifold. 
The AE conditions are preserved along such a Ricci flow (with the same coordinate system) \cite[Thm.\ 2.2]{L}. 
In this section, $(M^n,g_t)_{t\in I}$ denotes a spin, AE Ricci flow of order $\tau>\frac{n-2}{2}$ whose scalar curvature is nonnegative. The preservation of nonnegative scalar along Ricci flow follows from the maximum principle; see \cite[\S 3.3]{CLN}, for example.

With Theorem \ref{thm: existence and uniqueness of critical points} in hand, 
define $\kappa(t)$ to be the energy of the unique min-max critical point of $\mathscr{E}_{g(t)}$,
\begin{equation}
    \kappa(t)
    =\max_{f} \min_{\psi}\mathscr{E}_{g(t)}(\psi,f).
\end{equation}
This may be seen as an analogue of Perelman's $\lambda$-entropy for closed manifolds.
The main theorem of this section concerns the time derivative of $\kappa(g(t))$ along a Ricci flow $\partial_t g = -2 \Ric(g)$. 
Because the monotonicity theorem applies to the unmodified Ricci flow $\partial_tg=-2\Ric$, the proof uses the following $L^2$-orthogonality lemma.

\begin{lemma}\label{lem: L2 orthogonality of Ric_f and Hess_f}
If $(M^n,g)$ is an AE manifold of order $\tau>\frac{n-2}{2}$ and $f\in C^{2,\alpha}_{-\tau}(M)$ satisfies $\Scal_f=0$, then
\begin{equation}
    \int_M\langle \Hess_f,\Ric_f\rangle \,e^{-f}dV=0.
\end{equation}
\end{lemma}

\begin{proof}
The weighted Bianchi identity $\Div_f(\Ric_f-\frac{1}{2}\Scal_fg)=0$, which holds for all weighted manifolds, combined with the assumption $\Scal_f=0$ imply that $\Div_f(\Ric_f)=0$. Hence, the weighted divergence theorem (Appendix \ref{sec: weighted integration by parts}) and the fact that $\frac{1}{2}\mathscr{L}_{\nabla f}g=\Hess_f$ imply 
\begin{align}
    \int_M\langle \Hess_f,\Ric_f\rangle \,e^{-f}dV
    =\frac{1}{2}\int_M\langle \mathscr{L}_{\nabla f}g,\Ric_f\rangle \,e^{-f}dV
    =\lim_{\rho\to \infty}\int_{S_{\rho}}\Ric_f(\nabla f,\nu)\,e^{-f}dA.
\end{align}
Since the metric $g$ is AE of order $\tau$ and $f\in C^{2,\alpha}_{-\tau}$, the term $\Ric_f(\nabla f,\nu)$ decays like $r^{-2\tau-3}$. On the other hand, the area of $S_{\rho}$ is of order $\rho^{n-1}$. The assumption $\tau>\frac{n-2}{2}$ therefore implies that the above boundary term vanishes.
\end{proof}

\begin{proof}
[Proof of Theorem \ref{thm: monotonicity intro}]
The existence and uniqueness of $f$ and $\psi$ is the content of Theorem \ref{thm: existence and uniqueness of critical points}. 
Appendix \ref{sec: Time derivatives of weighted Witten spinors} proves the existence and regularity of their time derivatives. 
The variational formula (\ref{eqn: derivative of weighted Dirichlet plus Scal_f term}) implies
\begin{align}
    \frac{d}{dt}\int_M&|\nabla \psi|^2\,e^{-f}dV
        \nonumber\\ 
        &=\lim_{\rho\to \infty}\frac{1}{4}\int_{S_{\rho}}\biggl(
        2\langle \dot{g},T_{\psi}(\nu,\cdot)\rangle 
        +8\Rea\langle \dot{\psi},\nabla_{\nu}\psi\rangle
        +2\left(\frac{1}{2}\tr(\dot{g})-\dot{f}\right)\nabla_{\nu}|\psi|^2
        \\&\qquad\qquad
        -\langle \dot{g}(\nu,\cdot),\nabla |\psi|^2\rangle
        +|\psi|^2\biggl\langle\Div_f(\dot{g})-2\nabla\left(\frac{1}{2}\tr(\dot{g})-\dot{f}\right),\nu\biggr\rangle
        \biggr)e^{-f}dA.
        \nonumber
\end{align}
The first four boundary integrals vanish in the limit $\rho\to \infty$. Indeed, since $\psi-\psi_0\in C^{2,\alpha}_{-\tau}(M)$ and $|\psi|\to 1$ at infinity, $\nabla\psi=O(r^{-\tau-1})$ and $T_{\psi}(\nu,\cdot)=O(r^{-\tau-1})$.
Moreover, since $g$ is asymptotically flat of order $\tau$ and $f\in C^{2,\alpha}_{-\tau}(M)$, $\dot{g}=-2(\Ric+\Hess_f)$ is $O(r^{-\tau-2})$. Finally, by Proposition \ref{prop: regularity of time derivatives; weighted case}, $\dot{f}$ and $\dot{\psi}$ are $O(r^{-\tau})$. 
This shows that the first four terms in the above integrand all decay of order at least $r^{-2\tau -1}$; since $\tau>\frac{n-2}{2}$ and since the area of $S_{\rho}$ is of order $\rho^{n-1}$, said four terms all vanish in the limit $\rho\to \infty$.
Hence, only the last term in the integrand above contributes to the limit.

Since $|\psi|\to 1$ uniformly at infinity, the previous equation reduces to
\begin{align}
    \frac{d}{dt}\int_M|\nabla \psi|^2\,e^{-f}dV
        &=\lim_{\rho\to \infty}\frac{1}{4}\int_{S_{\rho}}\biggl\langle\Div_f(\dot{g})-2\nabla\left(\frac{1}{2}\tr(\dot{g})-\dot{f}\right),\nu\biggr\rangle
        \,e^{-f}dA 
\end{align}
Applying the weighted divergence theorem (Appendix \ref{sec: weighted integration by parts}) to the RHS and differentiating the equation $\Scal_f=0$ in time implies by (\ref{eqn: variation of Scal_f}) that 
\begin{align}
    \frac{d}{dt}\int_M|\nabla \psi|^2\,e^{-f}dV
        &=\frac{1}{4}\int_M\biggl(
        \Div_f^2(\dot{g})
        -2\Delta_f\left(\frac{1}{2}\tr(\dot{g})-\dot{f}\right)
        \biggr)e^{-f}dV
        \nonumber\\
        &=\frac{1}{4}\int_M\langle \dot{g},\Ric_f\rangle e^{-f}dV.
\end{align}
By Lemma \ref{lem: L2 orthogonality of Ric_f and Hess_f} and the fact that $\dot{g}=-2\Ric$, the last equation implies
\begin{equation}
    \frac{d}{dt}\int_M|\nabla \psi|^2\,e^{-f}dV
    =-\frac{1}{2}\int_M|\Ric_f|^2 e^{-f}dV.
\end{equation}
\end{proof}

\begin{remark}
In contrast to Perelman’s monotonicity for closed manifolds, which is proved by letting the weight $f$ evolve \emph{parabolically} backwards in time, the monotonicity formula (\ref{eqn: weighted Dirichlet monotonicity}) uses the fact that $f$ and $\psi$ solve the \emph{elliptic} equations $\Scal_f=0$ and $D_f\psi=0$ at each time; their time derivatives contribute only as boundary terms, which vanish due to the AE decay conditions.
\end{remark}

\subsection{Constancy of ADM mass.}

The ADM mass \cite{ADM} of an AE manifold $(M^n,g)$ is defined (up to a constant depending on dimension) by 
\begin{equation}\label{eqn: definition of mass}
    \mass(g)
    =\lim_{\rho\to \infty}\int_{S_{\rho}}(\partial_ig_{ij}-\partial_jg_{ii})\, \partial_j \intprod dV_{g}.
\end{equation}
The definition of mass involves a choice of AE coordinates, however, Bartnik \cite{Ba86} showed that if $\tau>(n-2)/2$ and the scalar curvature is integrable, then the mass is finite and independent of the choice of AE coordinates. If $n\leq 7$ or $(M^n,g)$ admits a spin structure, then the assumptions $\Scal\geq 0$, $\Scal\in L^1(M,g)$, and $\tau>\frac{n-2}{2}$, imply that $\mass(g)$ is nonnegative and is zero if and only if $(M^n,g)$ is isometric to $(\R^n,g_{\euc})$, by the positive mass theorem \cite{SY,W1}.

Witten argued that for any constant spinor $\psi_0$ on the end of $M$ with $|\psi_0|\to 1$ at infinity, there exists a harmonic spinor $\psi$ on $M$ which is asymptotic to $\psi_0$, in the sense that $\psi-\psi_0\in C^{2,\alpha}_{-\tau}(M)$. Such a spinor $\psi$ is called a Witten spinor because the ADM mass of $(M^n,g)$ is given by
\begin{equation}
    \mass(g)=4\int_M\left(|\nabla \psi|^2+\frac{1}{4}\Scal |\psi|^2\right) dV_g,
\end{equation}
which is called Witten's formula for the mass. A rigorous proof of the existence of Witten spinors is given by Parker-Taubes \cite{PT} and Lee-Parker \cite{LP}; their proofs were generalized to weighted AE manifolds in \cite{BO}.

\begin{proof}[Proof of Theorem \ref{thm: Constancy of ADM mass}]
Let $\psi$ be a Witten spinor, so $D\psi =0$. The variational formula (\ref{eqn: derivative of weighted Dirichlet plus Scal_f term}) applied with $f=\dot{f}=0$ reduces to
\begin{align}
    \frac{d}{dt}\int_M\left( 4|\nabla \psi|^2+\Scal|\psi|^2\right)dV
        &=\lim_{\rho\to \infty}\int_{S_{\rho}}\biggl(
        2\langle \dot{g},T_{\psi}(\nu,\cdot)\rangle 
        +8\Rea\langle \dot{\psi},\nabla_{\nu}\psi\rangle
        +\tr(\dot{g})\nabla_{\nu}|\psi|^2
        \\&\qquad\qquad
        -\langle \dot{g}(\nu,\cdot),\nabla |\psi|^2\rangle
        +|\psi|^2\langle\Div(\dot{g})-\nabla\tr(\dot{g}),\nu\rangle
        \biggr)dA.
        \nonumber
\end{align}
Using Proposition \ref{prop: Regularity of time derivatives; unweighted case}, the first four boundary integrals vanish in the limit $\rho\to \infty$ due to the asymptotic decay of these terms, as in the proof of the monotonicity theorem, Theorem \ref{thm: monotonicity intro}.
Hence, only the last term in the integrand above contributes to the limit.

The Bianchi identity $\Div(\Ric)=\frac{1}{2}\nabla \Scal$ applied to the previous equation yields
\begin{align}
    \frac{d}{dt}\int_M&\left( 4|\nabla \psi|^2+\Scal|\psi|^2\right)dV
        =\lim_{\rho\to \infty}\int_{S_{\rho}}
        |\psi|^2(\nabla_{\nu}\Scal)
        \,dA.
\end{align}
The latter boundary term vanishes since $|\psi|\to 1$ uniformly at infinity and by \cite[Lem.\ 11]{MS},
\begin{equation}
    \lim\limits_{\rho\to \infty}\int_{S_{\rho}}|\nabla \Scal|\,dA = 0.
\end{equation}
 
\end{proof}

\begin{remark}
The result of the above calculation agrees with that of \cite[p.\ 1843]{L}, where it was shown by different means that under the Ricci flow $\partial_tg=-2\Ric$, the ADM mass evolves by 
\begin{equation}
    \frac{d}{dt}\mass(g_t)=\lim_{\rho\to \infty}\int_{S_{\rho}}(\nabla_{\nu}\Scal) \, dA=0.
\end{equation}
\end{remark}
\appendix
\section{Appendix} \label{sec: appendix}
\subsection{Time derivatives of weighted Witten spinors.}\label{sec: Time derivatives of weighted Witten spinors}

The purpose of this appendix is to prove the existence and regularity of time derivatives of weighted Witten spinors along the Ricci flow. 
The argument is based on the existence theorem for ordinary differential equations in Banach spaces.

On an AE spin manifold $(M^n,g)$, the asymptotic coordinates define a positive function $r$, which equals the Euclidean distance to the origin on the end of $M$, and which can be extended to a smooth function which is bounded below by 1 on all of $M$. Using $r$, the weighted $C^k$ space $C^k_{\beta}(M)$ is defined for $\beta\in \R$ as the set of $C^k$ functions $u$ on $M$ for which the norm
\begin{equation*}
    \|u\|_{C^k_{\beta}}
        =\sum_{i=0}^k\sup_Mr^{-\beta +i}|\nabla^iu|
\end{equation*}
is finite. The weighted H\"older space $C^{k,\alpha}_{\beta}(M)$ is defined for $\alpha \in (0,1)$ as the set of $u\in C^k_{\beta}(M)$ for which the norm
\begin{equation*}
    \|u\|_{C^{k,\alpha}_{\beta}}
        =\|u\|_{C^k_{\beta}}
            +\sup_{x,y}\;(\min\{r(x),r(y)\})^{-\beta+k+\alpha}\frac{|\nabla^ku(x)-\nabla^ku(y)|}{d(x,y)^{\alpha}}
\end{equation*}
is finite. These definitions of weighted H\"older spaces coincide with those of \cite[\S 9]{LP}. In particular, the index $\beta$ denotes the {\it order of growth}: functions in $C^{k,\alpha}_{\beta}(M)$ grow at most like $r^{\beta}$. Note that the definitions of the weighted function spaces depend on the ``distance function'' $r$, and thereby on the choice of asymptotic coordinates. However, it is easy to see that $r$ is uniformly equivalent to the geodesic distance from an arbitrary fixed point in $M$ as $r\to \infty$, hence all choices of $r$ define equivalent norms. For the remainder of this appendix, fix $\alpha\in (0,1)$.

\begin{lemma}\label{lem: Dirac is isomorphism}
If $(M^n,g,f)$ is a weighted, AE manifold of order $\tau\in \left(\frac{n-2}{2},n-2\right)$ and $f\in C^{2,\alpha}_{-\tau}(M)$ satisfies $\Scal_f\geq 0$, then
\begin{equation*}
    D_f:C^{2,\alpha}_{-\tau}(M)\to C^{1,\alpha}_{-\tau-1}(M)
\end{equation*}
is an isomorphism. 
\end{lemma}

\begin{proof}
To show injectivity, suppose that $\psi\in C^{2,\alpha}_{-\tau}(M)$ satisfies $D_f\psi =0$. The weighted Lichnerowicz formula and integration by parts imply
\begin{align*}
    0
    &=\int_M\langle D^2_f\psi,\psi\rangle \,e^{-f} dV_g 
    =\int_M\left(-\langle \dL_f\psi,\psi\rangle +\frac{1}{4}\Scal_f |\psi|^2 \right)e^{-f} dV_g 
    =\int_M \left( |\nabla \psi|^2+\frac{1}{4}\Scal_f |\psi|^2\right)e^{-f} dV_g,
\end{align*}
because the boundary term vanishes if $\tau > \frac{n-2}{2}$.
Since $\Scal_f\geq 0$, this shows that $\nabla \psi=0$, so $\nabla |\psi|^2=0$. Thus $|\psi|$ is a constant, which is zero since $\psi$ vanishes at infinity. This proves $D_f$ is injective.

To show surjectivity, let $\xi \in C^{1,\alpha}_{-\tau-1}(M)$. 
Since
$
D_f^2:C^{2,\alpha}_{-\tau}(M)\to C^{0,\alpha}_{-\tau-2}(M)
$
is an isomorphism under the stated assumptions \cite[Lem.\ A.3]{BO}, there exists $\psi\in C^{2,\alpha}_{-\tau}(M)$ such that $D^2_f\psi=D_f\xi$. Then the spinor $\phi=D_f\psi-\xi$ satisfies $D_f\phi=0$ and lies in $C^{1,\alpha}_{-\tau-1}(M)$. Hence integration by parts as above implies $\phi=0$. Thus $D_f\psi=\xi$, showing that $D_f$ is surjective.
\end{proof}

For the remainder of this appendix, let $(M^n,g(t))_{t\in I}$ be an AE Ricci flow satisfying 
\begin{equation}
    \Scal\geq 0 \qquad \text{and} \qquad \frac{n-2}{2}<\tau < n-2.
\end{equation}
On such an AE Ricci flow, the ``distance function'' $r$ on $M$ is defined independently of time, since the AE condition is preserved by Ricci flow (with the same AE coordinates). 
Moreover, since the metrics $g(t)$, for $t\in I$, are uniformly equivalent by the AE condition, the $C^{k,\alpha}_{\beta}(M)$ weighted H\"older norm, defined with respect to $g(t)$, is equivalent, for all small $t\in I$, to the norm defined with respect to $g(0)$. This identification of the weighted H\"older spaces at different times along a Ricci flow is used implicitly in what follows.

\begin{lemma}
[Time derivative of $f$]\label{lem: time derivative of f}
For all small times along the Ricci flow, the boundary value problem $\Scal_f=0$ with $f\in C^{2,\alpha}_{-\tau}(M)$ admits a family of solutions which is $C^1$ in time, and whose time derivative $\dot{f}$ lies in $C^{2,\alpha}_{-\tau}(M)$.
\end{lemma}

\begin{proof}
The argument below proves the Lemma assuming a suitable solution of $\Scal_f=0$ exists at the initial time; existence of such a solution at the initial time follows from Theorem \ref{thm: existence and uniqueness of critical points}.

Define the time-dependent, linear operator
\begin{equation*}
    L_t:C^{2,\alpha}_{-\tau}(M)\to C^{0,\alpha}_{-\tau-2}(M)
    \qquad \qquad 
    L_tv=-4\Delta_{g(t)}v +\Scal_{g(t)}v.
\end{equation*}
When the choice of $t$ is clear from the context, $L_t$ is written as $L$ to simplify notation. The elliptic boundary value problem $\Scal_f=0$ with $f\in C^{2,\alpha}_{-\tau}$ can then be reformulated as
\begin{equation}\label{eqn: equation for v}
    Lv=-\Scal \qquad \text{for} \qquad v=e^{-f/2}-1\in C^{2,\alpha}_{-\tau}(M).
\end{equation}

The proof of the lemma is based on the existence of solutions to ordinary differential equations in the Banach space $C^{2,\alpha}_{-\tau}(M)$. If $v$ solves (\ref{eqn: equation for v}) along the Ricci flow and $v$ is $C^1$ in time, taking time derivatives of (\ref{eqn: equation for v}) implies
\begin{align}\label{eqn: intermediate dot v}
    \dot{L}v+L\dot{v}
    =-\dot{\Scal}.
\end{align}
The evolution equation for scalar curvature (\ref{eqn: variation of Scal_f}) along Ricci flow implies $\dot{\Scal}=\Delta\Scal+2|\Ric|^2$.
Above, the operator $\dot{L}:C^{2,\alpha}_{-\tau}(M)\to C^{0,\alpha}_{-\tau-2}(M)$ is the variation of $L_t$ along the Ricci flow, and is given by
\begin{equation}\label{eqn: variation of L}
    \dot{L}v
    =-8\langle \Ric,\Hess_v\rangle +(\Delta \Scal+2|\Ric|^2)v,
\end{equation}
the geometric quantities on the right above naturally being evaluated at time $t$; this formula follows from the variations of the Laplacian and scalar curvature along Ricci flow \cite[Lem.\ 2.30]{CLN}. Further, Li \cite{L} showed that if $g_{ij}(0)-\delta_{ij}\in C^k_{-\tau}$, then $g_{ij}(t)-\delta_{ij}\in C^{k-2}_{-\tau}$ for $t\geq 0$. 
In particular, if the metric $g(0)$ is initially smooth, then (\ref{eqn: intermediate dot v}) and (\ref{eqn: variation of L}) indeed are well-defined in $C^{0,\alpha}_{-\tau-2}$.

Equation (\ref{eqn: intermediate dot v}) can be inverted to obtain an ODE for the time derivative,
\begin{equation}\label{eqn: ODE for v}
    \dot{v}=-L^{-1}\dot{L}v-L^{-1}\dot{\Scal},
\end{equation}
because the operator $L_t$ is an isomorphism for each time. Indeed, a simple integration by parts argument using nonnegativity of $\Scal\in C^{0,\alpha}_{-\tau-2}$ and the decay conditions implies injectivity of $L$; surjectivity then follows from \cite[Thm.\ 9.2(d)]{LP}, since $\tau<n-2$.

Let $f(0)\in C^{2,\alpha}_{-\tau}(M)$ solve $\Scal_f=0$ at the initial time and define $v(0)=e^{-f(0)/2}+1$, so that $L_0v(0)=-\Scal_{g(0)}$ by (\ref{eqn: equation for v}). 
To prove the Lemma, it suffices to show that the above ODE in $C^{2,\alpha}_{-\tau}(M)$ admits a solution on some time interval $[0,\epsilon]$; indeed, if this is the case, then the fundamental theorem of calculus and (\ref{eqn: intermediate dot v}) imply that if $v(t)$ is defined by
\begin{equation}
    v(t)=v(0)+\int_0^t\dot{v}(s)\,ds,
\end{equation}
then $v(t)$ is $C^1$ in time and satisfies $L_tv(t)+\Scal_{g(t)}=L_0v(0)+\Scal_{g(0)}=0$. Hence if $v(0)$ solves (\ref{eqn: equation for v}), then $f(t):=-2\log(1+v(t))$ solves $\Scal_f=0$ on this time interval. Note that by continuity, if $1+v(0)$ is strictly positive, then $1+v(t)$ remains strictly positive for small time, so $f(t)$ is well-defined.

By the contraction mapping theorem, the existence for solutions of the ODE (\ref{eqn: ODE for v}) in $C^{2,\alpha}_{-\tau}(M)$ are guaranteed \cite[Thm.\ 9.4]{Co} on a time interval $[0,\epsilon]$, as long as the time-dependent vector field 
\begin{equation}
    X:[0,\epsilon]\times C^{2,\alpha}_{-\tau}(M)\to C^{2,\alpha}_{-\tau}(M)
    \qquad \qquad 
    X(t,v)=-L_t^{-1}\dot{L}_tv-L_t^{-1}\dot{\Scal}
\end{equation}
defined by the ODE is locally Lipschitz in the second variable and $\epsilon$ is chosen small enough (depending on the Lipschitz constant). The vector field $X$ is indeed locally Lipschitz, since
\begin{align}
    \|X(t,v_1)-X(t,v_2)\|_{C^{2,\alpha}_{-\tau}}
    &= \|L_t^{-1}\dot{L}_t(v_1-v_2)\|_{C^{2,\alpha}_{-\tau}}
    \nonumber\\
    &\leq \sup_{s\in [0,\epsilon]} \|L_s^{-1}\dot{L}_s\|_{\mathrm{op}}\|v_1-v_2\|_{C^{2,\alpha}_{-\tau}}.
\end{align}
The Lipschitz constant $L$ of the vector field $X$, given by
\begin{equation}
    L=\sup_{s\in [0,\epsilon]} \|L_s^{-1}\dot{L}_s\|_{\mathrm{op}},
\end{equation}
is finite because the time interval $[0,\epsilon]$ is compact, the operator $L_t$ varies smoothly along the Ricci flow, and the curvature and its derivatives are uniformly bounded along the flow by \cite{L}, hence $\dot{L}_s$ is also bounded from \eqref{eqn: variation of L}.
\end{proof}

Recall that the spin bundle depends on the Riemannian metric, though the spin bundles for different metrics are always isomorphic. 
In contrast to the generalized cylinder construction of Section \ref{sec: Spin geometry of generalized cylinders}, the existence and regularity of time derivatives of Witten spinors along an AE Ricci flow $(M^n,g(t))_{t\in I}$ is proven here using the time-dependent isometries \cite{BG} of the spin bundles 
\begin{equation*}
    \sigma_t:\Sigma_{g(t)}M\to \Sigma_{g(0)}M.
\end{equation*}
This allows for a convenient ODE formulation for a Witten spinor along the Ricci flow.

Since the AE coordinates are preserved along the flow, the notion of a spinor which is ``constant'' at infinity is defined independently of time. For the remainder of this appendix, fix a smooth spinor $\psi_0$ in the $g(0)$-spin bundle which is constant at infinity and of norm 1. 
Further, in the following two propositions, the H\"older space $C^{k,\alpha}_{\beta}(M)$ denotes the space of sections of the $g(0)$-spin bundle decaying suitably.

\begin{proposition}
[Time derivative of Witten spinor]
\label{prop: Regularity of time derivatives; unweighted case}
For all small times along the Ricci flow, the boundary value problem $D\psi=0$ with $\psi-\psi_0\in C^{2,\alpha}_{-\tau}(M)$ admits a family of solutions which is $C^1$ in time, and the time derivative $\dot{\psi}$ lies in $C^{2,\alpha}_{-\tau}(M)$.
\end{proposition}

\begin{proof}
The argument below proves the proposition assuming a Witten spinor exists at the initial time; existence at the initial time follows from Witten's proof of the positive mass theorem \cite{PT, LP}.

Define the time-dependent, linear operator
\begin{equation*}
    P_t:C^{2,\alpha}_{-\tau}(M)\to C^{1,\alpha}_{-\tau-1}(M)
    \qquad \qquad 
    P_t\psi=\sigma_tD_t\sigma_t^{-1}\psi.
\end{equation*}
When the choice of $t$ is clear from the context, $P_t$ is written as $P$ to simplify notation. 
Note that $\psi$ satisfies $P_t\psi=0$ and $\psi-\psi_0\in C^{2,\alpha}_{-\tau}$ if and only if $\sigma_t^{-1}\psi$ is a Witten spinor for the metric $g(t)$.
The elliptic boundary value problem $P\psi=0$ with $\psi-\psi_0\in C^{2,\alpha}_{-\tau}$ can be reformulated as
\begin{equation}\label{eqn: equation for xi}
    P\xi=-P\psi_0 \qquad \text{for} \qquad \xi=\psi-\psi_0\in C^{2,\alpha}_{-\tau}(M).
\end{equation}

The proof of the proposition is based on the existence of solutions to ordinary differential equations in the Banach space $C^{2,\alpha}_{-\tau}(M)$. If $\xi$ solves (\ref{eqn: equation for xi}) along the Ricci flow and $\xi$ is $C^1$ in time, taking time derivatives of (\ref{eqn: equation for xi}) and using that $\psi_0$ is time-independent implies
\begin{align}\label{eqn: intermediate dot xi}
    \dot{P}\xi+P\dot{\xi}
    =-\dot{P}\psi_0.
\end{align}
The evolution equation for the Dirac operator (\ref{lem: variation of Dirac operator}) along Ricci flow implies that the operator $\dot{P}:C^{2,\alpha}_{-\tau}(M)\to C^{1,\alpha}_{-\tau-1}(M)$, the variation of $P_t$ along the Ricci flow, is given by
\begin{equation}\label{eqn: variation of P}
    \dot{P}\xi
    =\sigma_t(\Ric(e_i)\cdot \nabla_i -\frac{1}{4}(\nabla \Scal) \cdot) \sigma_t^{-1}\xi,
\end{equation}
the curvatures and Clifford multiplication on the right above naturally being evaluated at time $t$.
Further, Li \cite{L} showed that if $g_{ij}(0)-\delta_{ij}\in C^k_{-\tau}$, then $g_{ij}(t)-\delta_{ij}\in C^{k-2}_{-\tau}$ for $t\geq 0$. In particular, if the metric $g(0)$ is initially smooth, then (\ref{eqn: intermediate dot xi}) and (\ref{eqn: variation of P}) indeed are well-defined in $C^{1,\alpha}_{-\tau-1}$.

Equation (\ref{eqn: intermediate dot xi}) can be inverted to obtain an ODE for the time derivative,
\begin{equation}\label{eqn: ODE for xi}
    \dot{\xi}=-P^{-1}\dot{P}(\xi+\psi_0),
\end{equation}
because the operator $P_t$ is an isomorphism for each time by Lemma \ref{lem: Dirac is isomorphism}.

Let $\psi(0)$ be a $g(0)$ Witten spinor asymptotic to $\psi_0$ and define $\xi(0)=\psi(0)-\psi_0$, so that $P_0\xi(0)=-P_0\psi_0$ by (\ref{eqn: equation for xi}). 
To prove the Lemma, it suffices to show that the above ODE in $C^{2,\alpha}_{-\tau}(M)$ admits a solution on some time interval $[0,\epsilon]$; indeed, if this is the case, then the fundamental theorem of calculus and (\ref{eqn: intermediate dot xi}) imply that if $\xi(t)$ is defined by
\begin{equation}
    \xi(t)=\xi(0)+\int_0^t\dot{\xi}(s)\,ds,
\end{equation}
then $\xi(t)$ is $C^1$ in time and satisfies $P_t\xi(t)+P_t\psi_0=P_0\xi(0)+P_0\psi_0=0$. 
Hence if $\xi(0)$ solves (\ref{eqn: equation for xi}), then $\psi(t):=\xi(t)+\psi_0$ is a $g(t)$ Witten spinor. 

By the contraction mapping theorem, the existence for solutions of the ODE (\ref{eqn: ODE for xi}) in $C^{2,\alpha}_{-\tau}(M)$ are guaranteed \cite[Thm.\ 9.4]{Co} on a time interval $[0,\epsilon]$, as long as the time-dependent vector field 
\begin{equation}
    Y:[0,\epsilon]\times C^{2,\alpha}_{-\tau}(M)\to C^{2,\alpha}_{-\tau}(M)
    \qquad \qquad 
    Y(t,\xi)=-P_t^{-1}\dot{P}_t(\xi+\psi_0),
\end{equation}
defined by the ODE is locally Lipschitz in the second variable and $\epsilon$ is chosen small enough (depending on the Lipschitz constant). The vector field $Y$ is indeed locally Lipschitz, since
\begin{align}
    \|Y(t,\xi_1)-Y(t,\xi_2)\|_{C^{2,\alpha}_{-\tau}}
    &= \|P_t^{-1}\dot{P}_t(\xi_1-\xi_2)\|_{C^{2,\alpha}_{-\tau}}
    \nonumber\\
    &\leq \sup_{s\in [0,\epsilon]} \|P_s^{-1}\dot{P}_s\|_{\mathrm{op}}\|\xi_1-\xi_2\|_{C^{2,\alpha}_{-\tau}}.
\end{align}
The Lipschitz constant $K$ of the vector field $Y$, given by
\begin{equation}
    K=\sup_{s\in [0,\epsilon]} \|P_s^{-1}\dot{P}_s\|_{\mathrm{op}},
\end{equation}
is finite because the time interval $[0,\epsilon]$ is compact, the operator $P_t$ varies smoothly along the Ricci flow, and the curvature and its derivatives are uniformly bounded along the flow by \cite{L}.
\end{proof}

\begin{proposition}
[Time derivative of weighted Witten spinor]\label{prop: regularity of time derivatives; weighted case}
Under the hypothesis of Theorem \ref{thm: monotonicity intro}, the time derivatives of $f$ and $\psi$ satisfy $\dot{f}\in C^{2,\alpha}_{-\tau}(M)$ and $\dot{\psi}\in C^{2,\alpha}_{-\tau}(M)$.
\end{proposition}

\begin{proof}
Because the scalar curvature is nonnegative, Witten's proof of the positive mass theorem implies the existence of an (unweighted) Witten spinor $\phi$. By the unitary equivalence \ref{eqn: unitary equivalence of Dirac and weighted Dirac} between the Dirac and weighted Dirac operators, $\psi=e^{-f/2}\phi$ is the weighted Witten spinor, and by Lemma \ref{lem: time derivative of f} and Proposition \ref{prop: Regularity of time derivatives; unweighted case}, the time derivative of $e^{-f/2}\phi$ exists and lies in $C^{2,\alpha}_{-\tau}$.
\end{proof}

\subsection{Weighted integration by parts formulas.}
\label{sec: weighted integration by parts}
Let $(M^n,g,f)$ be a weighted Riemannian manifold with boundary $\partial M$, whose outward normal is denoted $\nu$. The weighted divergence of a tensor $T$ is defined as 
\begin{equation*}
    \Div_f(T)=\Div(T)-T(\nabla f,\cdot).
\end{equation*}
The same definition applies when $T$ takes values in an auxiliary vector bundle equipped with a metric and compatible connection, like the spin bundle. 

The divergence theorem $\int_M\Div(X)\,dV=\int_{\partial M}\langle X,\nu\rangle \,dA$, along with the definition of the weighted divergence, implies the weighted divergence theorem
\begin{equation*}
    \int_M\Div_f(X)\,e^{-f}dV
    =\int_{\partial M}\langle X,\nu\rangle \,e^{-f}dA.
\end{equation*}
Applied to the vector field $uX$, for any function $u$ on $M$, the weighted divergence theorem implies
\begin{equation*}
    \int_Mu\Div_f(X)\,e^{-f}dV
        =-\int_M\langle \nabla u,X\rangle \,e^{-f}dV +\int_{\partial M}u\langle X,\nu\rangle \,e^{-f}dA.
\end{equation*}
Since the weighted Laplacian is defined as $\Delta_f=\Div_f\circ \nabla$, the above formula implies
\begin{equation*}
    \int_M\left(u\Delta_fv-(\Delta_fu)v\right)e^{-f}dV
        =\int_{\partial M} \left(u\nabla_{\nu}v-(\nabla_{\nu}u)v\right) e^{-f}dA.
\end{equation*}

The above discussion generalizes in a straightforward manner to higher-rank tensors: for any vector bundle valued $k$-tensor $T$ and $(k-1)$-tensor $S$, the weighted divergence theorem is
\begin{equation*}
    \int_M\langle \Div_f(T),S\rangle \,e^{-f}dV
    =-\int_M\langle T,\nabla S\rangle\,e^{-f}dV
    +\int_{\partial M}\langle T(\nu,\cdot),S\rangle\,e^{-f}dA.
\end{equation*}
This follows from Stokes theorem applied to the $(n-1)$-form $\alpha =\iota_X(dV_g)$, where $X$ is the vector field $X=\langle T(e_i,\cdot),S\rangle e^{-f}e_i$; indeed, with these choices, it follows that
\begin{equation*}
    d\alpha = \Div(X)dV=\left(\langle \Div_f(T),S\rangle +\langle T,\nabla S\rangle\right)e^{-f}dV.
\end{equation*}
For a symmetric 2-tensor $T$ on $M$, the weighted divergence theorem simplifies: symmetry of $T$ and the definition of the Lie derivative of the metric implies, for every vector field $X$ on $M$, 
\begin{equation*}
    \int_M\langle \Div_f(T),X\rangle \,e^{-f}dV
    =-\frac{1}{2}\int_M\langle T,\mathscr{L}_Xg\rangle\,e^{-f}dV
    +\int_{\partial M}T(X,\nu)\,e^{-f}dA.
\end{equation*}

\newcommand{\Addresses}{{
  \bigskip
  \bigskip
  \small
    \textsc{Department of Mathematics}\par\nopagebreak
    \textsc{Massachusetts Institute of Technology}\par\nopagebreak
    \textsc{77 Massachusetts Avenue} \par\nopagebreak
    \textsc{Cambridge, MA 02139} 
    
    \medskip
    \medskip
    
    \textit{Correspondence to be sent to:} 
    \texttt{juliusbl@mit.edu}

}}

\Addresses

\end{document}